\documentclass[12pt]{amsart}

\usepackage{latexsym, amssymb, amsmath}

\usepackage{amsfonts, graphicx}
\usepackage{graphicx,color}
\newcommand{\be}{\begin{equation}}
\newcommand{\ee}{\end{equation}}
\newcommand{\beq}{\begin{eqnarray}}
\newcommand{\eeq}{\end{eqnarray}}

\newtheorem{claim}{Claim}[section]

\newtheorem{thm}{Theorem}[section]

\newtheorem{lma}{Lemma}[section]
\newtheorem{prop}{Proposition}[section]
\newtheorem{cor}{Corollary}[section]
\newtheorem{defn}{Definition}[section]
\theoremstyle{remark}
\newtheorem{rem}{Remark}[section]
\numberwithin{equation}{section}
\makeatletter
\newcommand*\owedge{\mathpalette\@owedge\relax}
\newcommand*\@owedge[1]{%
  \mathbin{%
    \ooalign{%
      $#1\m@th\bigcirc$\cr
      \hidewidth$#1\m@th\wedge$\hidewidth\cr
    }%
  }%
}
\makeatother

\def\det{\mathrm{det}}

\def\p{\partial}

\def\C{\mathcal{C}}
\def\R{\mathbb{R}}

\def\p{\partial}
\def\lf{\left}
\def\ri{\right}
\def\e{\epsilon}
\def\ol{\overline}
\def\R{\Bbb R}

\def\Pi{\overline{\displaystyle{\mathbb{II}}}}

\def\heat{\lf(\Delta -\frac{\p}{\p t}\ri)}
\def\K{K\"ahler }

\def\heat{\lf(\frac{\p}{\p t}-\Delta\ri)}

\def\lf{\left}
\def\ri{\right}

\def\ol{\overline}

\def\e{\epsilon}
\def\p{\partial}

\def\C{\Bbb C}
\def\R{\Bbb R}

\def\K{K\"ahler\ }
\def\be{\begin{equation}}
\def\ee{\end{equation}}

\def\ijb{i\bar j}

\def\bee
{\begin{equation*}}

\def\eee{\end{equation*}}

\def\bee{\begin{equation*}}
\def\eee{\end{equation*}}

\def\ol{\overline}
\def\e{\epsilon}
\def\lf{\left}
\def\heat{\lf(\frac{\p}{\p t}-\Delta\ri)}
\def\ijb{i\bar{j}}
\def\ri{\right}

\def\p{\partial}

\def\jbar{{\bar\jmath}}

\def\K{K\"ahler }
\def\KR{K\"ahler-Ricci }

\def\be{\begin{equation}}
\def\ee{\end{equation}}
\def\ol{\overline}
\def\lf{\left}
\def\ri{\right}

\def\e{\epsilon}

\def\ijb{{i\jbar}}

\def\p{\partial}

\def\C{\Bbb C}

\def\p{\partial}

\def\p{\partial}

\def\C{\Bbb C}

\def\bee{\begin{equation*}}
\def\eee{\end{equation*}}

\def\ol{\overline}
\def\e{\epsilon}
\def\lf{\left}
\def\heat{\lf(\frac{\p}{\p t}-\Delta\ri)}
\def\ijb{i\bar{j}}
\def\ri{\right}

\def\p{\partial}

\def\jbar{{\bar\jmath}}

\def\K{K\"ahler }
\def\KR{K\"ahler-Ricci }

\def\be{\begin{equation}}
\def\ee{\end{equation}}
\def\ol{\overline}
\def\lf{\left}
\def\ri{\right}

\def\e{\epsilon}

\def\ijb{{i\jbar}}

\def\p{\partial}

\def\C{\Bbb C}

\def\p{\partial}

\def\p{\partial}

\def\C{\Bbb C}

\def\bee{\begin{equation*}}
\def\eee{\end{equation*}}

\def\ol{\overline}
\def\e{\epsilon}
\def\lf{\left}
\def\heat{\lf(\frac{\p}{\p t}-\Delta\ri)}
\def\ijb{i\bar{j}}
\def\ri{\right}

\def\p{\partial}

\def\jbar{{\bar\jmath}}

\def\K{K\"ahler }
\def\KR{K\"ahler-Ricci }

\def\be{\begin{equation}}
\def\ee{\end{equation}}
\def\ol{\overline}
\def\lf{\left}
\def\ri{\right}

\def\e{\epsilon}

\def\ijb{{i\jbar}}

\def\p{\partial}

\def\C{\Bbb C}

\def\p{\partial}

\def\p{\partial}

\def\C{\Bbb C}

\def\bee{\begin{equation*}}
\def\eee{\end{equation*}}

\def\ol{\overline}
\def\e{\epsilon}
\def\lf{\left}
\def\heat{\lf(\frac{\p}{\p t}-\Delta\ri)}
\def\ijb{i\bar{j}}
\def\ri{\right}

\def\p{\partial}

\def\jbar{{\bar\jmath}}

\def\K{K\"ahler }
\def\KR{K\"ahler-Ricci }

\def\be{\begin{equation}}
\def\ee{\end{equation}}
\def\ol{\overline}
\def\lf{\left}
\def\ri{\right}

\def\e{\epsilon}

\def\ijb{{i\jbar}}

\def\p{\partial}

\def\C{\Bbb C}

\def\p{\partial}

\def\p{\partial}

\def\C{\Bbb C}

\def\KE{K\"ahler-Einstein }

\def\ddb{\sqrt{-1}\partial\bar\partial}

\begin{document}

\title[]
{second Ricci flow on noncompact Hermitian manifolds}

 \author{Man-Chun Lee}
\address[Man-Chun Lee]{Department of Mathematics, Northwestern University, Evanston, IL 60208, USA}
\email{mclee@math.northwestern.edu}

\renewcommand{\subjclassname}{
  \textup{2010} Mathematics Subject Classification}

\date{\today}

\begin{abstract}
In this work, we first establish short time existence and Shi's type estimate of second Ricci flow on complete noncompact Hermitian manifolds. As an application, we use the second Ricci flow to discuss the existence of \KE metric on complete noncompact Hermitian manifolds. 
\end{abstract}


\maketitle

\markboth{Man-Chun Lee}{second Ricci flow on noncompact Hermitian manifolds}
\section{introduction}
Let $(M,g,J)$ be a complete complex manifold and $g$ is a Hermitian metric. $g$ is Hermitian if $g$ is a Riemannian metric and also satisfies $g(X,Y)=g(JX,JY)$ for all $X,Y\in TM$. When $M$ is compact \K or complete noncompact \K with suitable curvature, it was shown that the Hamilton Ricci flow will preserve the \K condition \cite{Cao1985,Shi1997,HuangTam2018}. The \KR flow was then found to be very powerful in the study of  geometrical classification in \K geometry.  However, when $g$ is non-K\"ahler, generally the Ricci flow will no longer preserve the Hermitian condition. Inspired by this, one may ask if there is any alternative parabolic flow which also preserves the Hermitian structure. In \cite{Gill2011}, Gill introduced a Hermitian flow called the Chern-Ricci flow which aims to study the existence of Hermitian metric with flat Chern-Ricci curvature.

In this work, we are interested in special case of Hermitian flow on complete noncompact Hermitian manifolds which was first introduced by Streets and Tian \cite{StreetTian2011}: 
$$\frac{\partial}{\partial t}g =-S_g,\;\;g(0)=g_0$$
where $S_g$ is the second Ricci curvature with respect to the Chern connection of $g$. This flow was also appeared in \cite{LiuYang2012} which aim to study Hermitian Einstein metric on complex manifolds. In this article, we will call it the Hermitian Ricci flow. In fact, in \cite{StreetTian2011}, a more general Hermitian flow was introduced in which the direction of the deformation may involve torsion terms $Q(T)$. The Hermitian curvature flow was also found to be useful in the study of Hermitian geometry.
 In \cite{StreetTian2010}, they initiated a program of studying a particular choice of $Q$ in which the flow preserves the pluriclosed condition. More recently Ustinovskiy \cite{Yury2016} also showed that for a different choice of $Q$ the flow will preserve the nonnegativity of bisectional curvature which leads to an extensions of the classical Frankel conjecture to quasi-positive case. Motivated by the work in \cite{Liu2014,BohmWilking2007}, the author study the first Ricci curvature along the flows with $Q=0$ and use it to show that compact Hermitian manifolds with quasi-negative bisectional curvature have ample canonical line bundle. By the celebrated work of Yau and Aubin, \cite{Aubin1976,Yau1978}, it is equivalent to say that  there is a \KE metric with negative scalar curvature.

In this article, we wish to study the existence of \KE metric  on complete noncompact Hermitian manifolds. To extend the work in \cite{Lee2018}, we first develop some foundational results on the short-time existence of the Hermitian flow under some reasonable assumption. In particular, we have the following short-time existence result. 
\begin{thm}\label{MAIN-EX}
Suppose $(M,g_0)$ is a complete noncompact Hermitian manifold with
$$\sup_M |Rm|+|T|^2+|\nabla T|<+\infty,$$
then there is a short-time solution $g(t)$ on $M\times [0,\tau]$ to 
$$\partial_t g_{i\bar j}=-S_{i\bar j},\;\;g(0)=g_0.$$
Moreover, the solution $g(t)$ has bounded Chern curvature and torsion on $[0,\tau]$.
\end{thm} 
\begin{rem}For general $Q$, the corresponding Hermitian curvature will also admit short-time solution under the assumption made above which is a non-compact version of Streets-Tian's work \cite{StreetTian2011}. In fact under boundedness of Chern curvature and Torsion, the boundedness on $\nabla T$ is equivalent to boundedness of Riemannian curvature $Rm^L$. When $Q=0$, this can be replaced by existence of good exhaustion function on $M$. We refer readers to Theorem \ref{Improved-short-time} for detailed statement.\end{rem}


Next, we wish to apply the flow to study existence of \KE metric on negatively curved Hermitian manifolds.
 In this work, we are interested in the case when a complete noncompact Hermitian manifold has quasi-negative curvature. Using the existence of the Hermitian Ricci flow together with Shi-type estimates, we  extend the result in \cite{Lee2018} to complete noncompact case.  Our main result is the following.
\begin{thm}\label{main-main}
Let $(M,g_0)$ be a complete noncompact Hermitian manifold with bounded Riemannian curvature, Chern curvature and torsion. Suppose $g_0$ has non-positive Chern bisectional curvature and quasi-negative first Ricci curvature, then $M$ supports a \K metric which maybe incomplete. Furthermore, if the first Ricci curvature is uniformly negative outside a compact set, then $M$ supports a complete \KE metric $g_{KE}=-Ric(g_{KE})$ with bounded curvature.
\end{thm}

It is unclear to the author whether $M$ supports a \textit {complete} \K metric in the quasi-negative case. When $M$ is noncompact, the existence of complete \KE metric was studied by various authors, see for example \cite{ChengYau1980,Wu2008} based on assumptions on the Ricci curvature and \cite{WuYau2017,HuangLeeTamTong2018} based on the negativity of holomorphic sectional curvature. Theorem \ref{main-main} is different from the previous results since the K\"ahlerity is a priori unknown.

The paper is organized as follows: In section 2, we recall some preliminary definitions and formula about the Chern connection. In section 3, we will derive evolution equations for the Hermitian Ricci flow. In section 4, we will derive some a-priori estimates for the Hermitian Ricci flow. In section 5, 6, we will prove the general short-time existence to Hermitian manifolds with bounded Riemannian curvature, Chern curvature and torsion. In section 7, 8, we will give a proof of Theorem \ref{main-main}.

{\it Acknowledgement}: The author would also like to thank the referee for useful comments.

\section{Chern connection}\label{background}
In this section, we collect some useful formulas for the Chern connection. Those materials can be found in \cite{ShermanWeinkove2013}. Let $(M,g)$ be a Hermitian manifold. The {\it Chern connection} of $g$ is defined as follows: In local holomorphic coordinates $z^i$,
 for a vector field $X_i\p_i$, where $\p_i:=\frac{\p}{\p z^i}$, $\p_{\bar i}=\frac{\p}{\p \bar z^i}$,
 $$\nabla_iX^k=\p_{i}X^k+\Gamma_{ij}^kX^j;\ \nabla_{\bar i}X^k=\p_{\bar i}X^k.
 $$
 For a $(1,0)$ form $a=a_idz^i$,
 $$
 \nabla_ia_j=\p_i a_j-\Gamma_{ij}^ka_k; \ \nabla_{\bar i}a_j=\p_{\bar i}a_j.
 $$
 Here $\nabla_i:=\nabla_{\p_i}$, etc. $\Gamma$ are the coefficients of $\nabla$,
with
$$\Gamma_{ij}^k=g^{k\bar l}\partial_i g_{j\bar l}.$$
Noted that Chern connection is a connection such that $\nabla g=\nabla J=0$ and the torsion has no $(1,1)$ component.   The {\it torsion} of $g$ is defined to be
$$T_{ij}^k=\Gamma_{ij}^k-\Gamma_{ji}^k.$$
We remark that $g$ is \K if and only if $T=0$. Define the \textit{Chern curvature tensor} of $g$ to be
$$R_{i\bar jk}\,^l=-\partial_{\bar j}\Gamma_{ik}^l.$$
We raise and lower indices by using metric $g$. {Direct computations show:
$$
\ol{R_{i\bar jk\bar l}}=R_{j\bar il\bar k}.
$$
In this note, we will use $Rm$ to denote curvature tensor with respect to the Chern-connection while $Rm^L$ will denote the Riemannian curvature tensor.}

The Chern-Ricci curvature is defined by
$$R_{i\bar j}=g^{k\bar l}R_{i\bar j k\bar l}=-\partial_i \partial_{\bar j}\log \det g.$$
Note that if $g$ is not K\"ahler, then $R_\ijb$ may not equal to $g^{k\bar l}R_{ k\bar li\bar j}=S_{i\bar j}$. In some content, $Ric$ is  sometimes called first Ricci curvature while $S$ is called the second Ricci.
\begin{lma}
The commutation formulas for the Chern curvature are given by
\begin{align*}
[\nabla_i,\nabla_{\bar j}]X^l=R_{i\bar j k}\,^l X^k,\quad\quad [\nabla_i,\nabla_{\bar j}]a_k=-R_{i\bar j k}\,^l a_l;\\
[\nabla_i,\nabla_{\bar j}]X^{\bar l}=-R_{i\bar j}\,^{\bar l}\,_{\bar k} X^{\bar k},\quad\quad [\nabla_i,\nabla_{\bar j}]a_{\bar k}=R_{i\bar j}\,^{\bar l}\,_{\bar k} a_{\bar l}.
\end{align*}
\end{lma}

When $g$ is not K\"ahler, the Bianchi identities maybe fail. The failure can be measured by the torsion tensor.

\begin{lma}\label{l-Chern-connection-1}
In a holomorphic local coordinates, let $T_{ij\bar k}=g_{p\bar k}T_{ij}^p$,  we have
\begin{align*}
R_{i\bar jk\bar l}-R_{k\bar ji\bar l}&=-\nabla_{\bar j}T_{ik\bar l},\\
R_{i\bar jk\bar l}-R_{i\bar lk\bar j}&=-\nabla_{i}T_{\bar j\bar lk},\\
R_{i\bar jk\bar l}-R_{k\bar li\bar j}&=-\nabla_{\bar j}T_{ik\bar l}-\nabla_kT_{\bar j\bar li}=-\nabla_iT_{\bar j\bar lk}-\nabla_{\bar l}T_{ik\bar j},\\
\nabla_pR_{i\bar jk\bar l}-\nabla_iR_{p\bar jk\bar l}&=-T_{pi}^rR_{r\bar jk\bar l},\\
\nabla_{\bar q}R_{i\bar jk\bar l}-\nabla_{\bar j}R_{i\bar qk\bar l}&=-T_{\bar q\bar j}^{\bar s}R_{i\bar sk\bar l}.
\end{align*}
\end{lma}

It can be checked easily that for $X,Y\in T^{1,0}M$, $R(X,\bar X, Y,\bar Y)$ is real-valued. We consider the following curvature condition.
\begin{defn}
We say that $(M,g)$ has holomorphic bisectional curvature bounded above by a function $\kappa(x)$ if for any $x\in M$, $X,Y\in T^{1,0}_xM$,
$$R(X,\bar X,Y,\bar Y)\leq \kappa B(X,\bar X, Y,\bar Y) $$
where $B_{i\bar jk\bar l}=g_{i\bar j} g_{k\bar l}+g_{i\bar l}g_{k\bar j}$. 
\end{defn}

Here we should remark that our notation of bisectional curvature is slightly different from that in \cite{LeeWang2005}.

\begin{defn}
We say that $(M,g)$ has Chern-Ricci curvature bounded above by a function $\kappa(x)$ if for any $p\in M$, $X\in T_p^{1,0}M$, 
$$Ric(X,\bar X)\leq \kappa(p) g(X,\bar X).$$
If $\kappa$ is non-positive and negative at some point $z\in M$, then we say that $g$ has quasi-negative Chern-Ricci curvature.
\end{defn}
In this note, all the curvature tensor $Rm$ will be referring to the curvature tensor with respect to Chern connection.

\section{Evolution equations for the Hermitian Ricci flow}
In this section, we will discuss a special type of Hermitian Ricci flow introduced by \cite{StreetTian2011} with $Q\equiv 0$: 
\be\label{HRF}
\left\{
  \begin{array}{ll}
    \frac{\p}{\p t}g_\ijb=  &  -S_{i\bar j}; \\
    g(0)=  &g_0.
  \end{array}
\right.
\ee
Here $S_{i\bar j}=g^{k\bar l}R_{k\bar li\bar j}$ is the second Ricci curvature with respect to the Chern connection while the Chern-Ricci curvature (or first Ricci curvature) is defined by $R_{i\bar j}=g^{k\bar l}R_{i\bar jk\bar l}$. It coincides with the Chern-Ricci curvature if the metric is K\"ahler. However they are different  in general.

To begin with, we would like to point out that the Hermitian Ricci flow is indeed a parabolic system which is in a similar form as the Ricci DeTurck flow which was shown explicitly by Shi in \cite[Lemma 2.1]{Shi1989}.
\begin{lma}\label{para-sys}
In local coordinate, we have 
\begin{equation}
\begin{split}
\frac{\partial}{\partial t}g_{i\bar j}&=\frac{1}{2}g^{k\bar l}\left( \tilde\nabla_k \tilde\nabla_{\bar l}+ \tilde\nabla_{\bar l}\tilde\nabla_k\right)g_{i\bar j}-g^{k\bar l} g^{p\bar q}(\tilde\nabla_k g_{i\bar q})(\tilde\nabla_{\bar l}g_{p\bar j})\\
&\quad  -\frac{1}{2}\left(g^{k\bar l}g_{p\bar j} \tilde R_{k\bar li}\,^p+g^{k\bar l}g_{i\bar q}\tilde R_{k\bar l}\,^{\bar q}_{\bar j}\right).
\end{split}
\end{equation}
Here $\tilde \nabla$ and $\tilde R$ denotes the Chern connection and the Chern curvature of $g_0$ respectively.
\end{lma}
\begin{proof}
\begin{equation}
\begin{split}
-S_{i\bar j}&=-g^{k\bar l} R_{k\bar li\bar j}\\
&=g^{k\bar l} g_{p\bar j}(\partial_{\bar l} \Gamma^p_{ki}-\partial_{\bar l} \tilde\Gamma^p_{ki})-g^{k\bar l}g_{p\bar j} \tilde R_{k\bar li}\,^p\\
&=g^{k\bar l} g_{p\bar j}\tilde\nabla_{\bar l}\left(g^{p\bar q}\tilde\nabla_k g_{i\bar q} \right)-g^{k\bar l}g_{p\bar j} \tilde R_{k\bar li}\,^p\\
&=g^{k\bar l}\tilde\nabla_{\bar l}\tilde\nabla_k g_{i\bar j}-g^{k\bar l} g^{p\bar q}(\tilde\nabla_k g_{i\bar q})(\tilde\nabla_{\bar l}g_{p\bar j})-g^{k\bar l}g_{p\bar j} \tilde R_{k\bar li}\,^p\\
&=\frac{1}{2}g^{k\bar l}\left( \tilde\nabla_k \tilde\nabla_{\bar l}+ \tilde\nabla_{\bar l}\tilde\nabla_k\right)g_{i\bar j}-g^{k\bar l} g^{p\bar q}(\tilde\nabla_k g_{i\bar q})(\tilde\nabla_{\bar l}g_{p\bar j})\\
&\quad -\frac{1}{2}\left(g^{k\bar l}g_{p\bar j} \tilde R_{k\bar li}\,^p+g^{k\bar l}g_{i\bar q}\tilde R_{k\bar l}\,^{\bar q}_{\bar j}\right).
\end{split}
\end{equation}
\end{proof}

\begin{lma}\label{evo-trace-1}
Suppose $g(t)$ is a solution to the Hermitian Ricci flow, then we have 
\begin{align}
\heat tr_g g_0 =-(g_0)_{p\bar q}g^{k\bar l} g^{i\bar j} \Psi^p_{ki}\Psi^{\bar q}_{\bar l \bar j}+g^{k\bar l} g^{i\bar j} \hat R_{k\bar l i\bar j}
\end{align}
where $\Psi=\Gamma_{g_0}-\Gamma_g$ denotes the difference between the Chern connection of $h$ and that of $g$ while $\hat R$ is the Chern curvature of $g_0$. 
\end{lma}
\begin{proof}
Differentiate it with respect to $t$, we have 
\begin{equation}\label{trace-1}
\begin{split}
\frac{\partial }{\partial t} (g^{i\bar j}(g_0)_{i\bar j})
&=S^{i\bar j}(g_0)_{i\bar j}.
\end{split}
\end{equation}

On the other hand, 
\begin{equation}\label{trace-2}
\begin{split}
\Delta (g^{i\bar j}(g_0)_{i\bar j})
&=h_{p\bar q}g^{k\bar l} g^{i\bar j} \Psi^p_{ki}\Psi^{\bar q}_{\bar l \bar j}-g^{k\bar l} g^{i\bar j} \hat R_{k\bar l i\bar j}+S^{i\bar j}h_{i\bar j}.
\end{split}
\end{equation}
The conclusion follows immediately by adding \eqref{trace-1} and \eqref{trace-2} together.
\end{proof}

\begin{lma}\label{firstorder-diff}
Suppose $(M,g(t))$ is a soliution to \eqref{HRF}, then the tensor $\Psi_{ij}^k=\tilde\Gamma_{ij}^k-\Gamma_{ij}^k$ satisfies
\begin{align*}\heat |\Psi|^2&=-|\nabla \Psi|^2-|\bar\nabla \Psi|^2+2{\bf Re}\left[ g^{i\bar j} g^{k\bar l}g_{r\bar s}\Psi^{\bar s}_{\bar j\bar l}(g^{p\bar q}T^a_{pi}R_{a\bar qk}\,^r+g^{p\bar q}\nabla_p\tilde R_{i\bar qk}\,^r)\right].
\end{align*}
Here $\tilde \Gamma$and $\tilde R$ denotes the Chern connection and Chern curvature with respect to $g_0$ and the norm is calculated using the evolving metric $g(t)$.

\end{lma}
\begin{proof}
First noted that 
\begin{align*}
\partial_t \Psi_{ij}^k
&=g^{k\bar l} \nabla_i S_{j\bar l}.
\end{align*}
On the other hand,
\begin{align*}
\Delta |\Psi|^2&=g^{p\bar q} \nabla_p \nabla_{\bar q} (g^{i\bar j}g^{k\bar l}g_{r\bar s}\Psi_{ik}^r\Psi_{\bar j\bar l}^{\bar s})\\
&=g^{p\bar q} g^{i\bar j}g^{k\bar l}g_{r\bar s} \nabla_p\nabla_{\bar q} (\Psi^r_{ik}\Psi^{\bar s}_{\bar j\bar l})\\
&=g^{p\bar q} g^{i\bar j}g^{k\bar l}g_{r\bar s} \left(\nabla_p \Psi^r_{ik}\cdot \nabla_{\bar q}\Psi^{\bar s}_{\bar j\bar l}+\nabla_{\bar q} \Psi^r_{ik}\cdot \nabla_{ p}\Psi^{\bar s}_{\bar j\bar l} \right)\\
&\quad+g^{p\bar q} g^{i\bar j}g^{k\bar l}g_{r\bar s} \left(\nabla_p\nabla_{\bar q} \Psi^r_{ik}\cdot \Psi^{\bar s}_{\bar j\bar l}+\nabla_{ p}\nabla_{\bar q}\Psi^{\bar s}_{\bar j\bar l} \cdot  \Psi^r_{ik}\right)\\
&=|\nabla \Psi|^2+|\bar\nabla \Psi|^2+g^{p\bar q} g^{i\bar j}g^{k\bar l}g_{r\bar s} \left(\nabla_p\nabla_{\bar q} \Psi^r_{ik}\cdot \Psi^{\bar s}_{\bar j\bar l}+\nabla_{ p}\nabla_{\bar q}\Psi^{\bar s}_{\bar j\bar l} \cdot  \Psi^r_{ik}\right)\\
&=|\nabla \Psi|^2+|\bar\nabla \Psi|^2+g^{p\bar q}g^{i\bar j}g^{k\bar l}g_{r\bar s}\Psi^{\bar s}_{\bar j\bar l}\nabla_p(R_{i\bar qk}\,^r-\tilde R_{i\bar qk}\,^r)\\
&\quad +g^{p\bar q}g^{i\bar j}g^{k\bar l}g_{r\bar s}\Psi^r_{ik} \overline{\nabla_q (R_{ j\bar p l}\,^{ s}-\tilde R_{j\bar p l}\,^{s})}\\
&\quad +g^{i\bar j}g^{k\bar l}g_{r\bar s}\Psi^r_{ik} \left(S^{\bar s}\,_{\bar q} \Psi^{\bar q}_{\bar j\bar l}-S^{\bar q}\,_{\bar j}\Psi^{\bar s}_{\bar q\bar l}-S^{\bar q}\,_{\bar l}\Psi^{\bar s}_{\bar j\bar q}\right)\\
&=|\nabla \Psi|^2+|\bar\nabla \Psi|^2+2Re\left[g^{i\bar j}g^{k\bar l} g_{r\bar s} \Psi^{\bar s}_{\bar j\bar l}(\nabla_i S^r_k-g^{p\bar q}T^a_{pi}R_{a\bar qk}\,^r-g^{p\bar q}\nabla_p\tilde R_{i\bar qk}\,^r)\right]\\
&\quad +g^{i\bar j}g^{k\bar l}g_{r\bar s}\Psi^r_{ik} \left(S^{\bar s}\,_{\bar q} \Psi^{\bar q}_{\bar j\bar l}-S^{\bar q}\,_{\bar j}\Psi^{\bar s}_{\bar q\bar l}-S^{\bar q}\,_{\bar l}\Psi^{\bar s}_{\bar j\bar q}\right)
\end{align*}
where we have used the fact that
$$g^{p\bar q}\nabla_p R_{i\bar q k}\,^r=\nabla_i S_{k}^r-g^{p\bar q}T_{pi}^s R_{s\bar qk}\,^r.$$

Therefore, we can conclude that 
\begin{align*}
\heat |\Psi|^2&=-|\nabla \Psi|^2-|\bar\nabla \Psi|^2+2{\bf Re}\left[ g^{i\bar j} g^{k\bar l}g_{r\bar s}\Psi^{\bar s}_{\bar j\bar l}(g^{p\bar q}T^a_{pi}R_{a\bar qk}\,^r+g^{p\bar q}\nabla_p\tilde R_{i\bar qk}\,^r)\right]
\end{align*}
\end{proof}

Now we collect  the evolution equation for the Chern curvature tensor $R_{i\bar j k\bar l}$ which can be found in \cite[Section 6]{StreetTian2011}. 
\begin{lma}Suppose $g(t)$ is a solution to the Hermitian Ricci flow, we have
\begin{equation}\label{Rm-evo}
\begin{split}
\partial_t R_{i\bar j k\bar l}
&=\Delta R_{i\bar jk\bar l}+g^{r\bar s}\Big[T^p_{ri}\,\nabla_{\bar s} R_{p\bar jk\bar l}+T^{\bar q}_{\bar s \bar j}\,\nabla_r R_{i\bar qk\bar l}+T^p_{ri}T^{\bar q}_{\bar s\bar j}R_{p\bar q k\bar l}\\
&\quad  +R_{i\bar j r}\,^pR_{p\bar sk\bar l}+R_{r\bar j k}\,^p R_{i\bar s p\bar l} -R_{r\bar jp\bar l} R_{i\bar s k}\,^p\Big]\\
&\quad -\frac{1}{2}\left[S^p_i R_{p\bar jk\bar l}+S^p_k R_{i\bar jp\bar l}+S^{\bar q}_{\bar j}R_{i\bar qk\bar l}+S^{\bar q}_{\bar l}R_{i\bar j k\bar q} \right].
\end{split}
\end{equation}
\end{lma}

By tracing $k$ and $l$, we arrive at the evolution equation of the Chern-Ricci curvature (or first Ricci curvature). For detailed computation, we refer to \cite{Lee2018}.
\begin{lma}\label{Ricci-evo}Suppose $g(t)$ is a solution to the Hermitian Ricci flow, we have the following evolution equation for the Chern-Ricci curvature.
\begin{equation}
\begin{split}\partial_t R_{i\bar j}=&\Delta R_{i\bar j}+g^{r\bar s} (T^p_{ri} \nabla_{\bar s} R_{p\bar j}+T^{\bar q}_{\bar s\bar j}\nabla_r R_{i\bar q}+T^p_{ri}T^{\bar q}_{\bar s\bar j}R_{p\bar q})\\
&+R_{i\bar jk}\,^p R_{p}\,^k-\frac{1}{2}\left[S^p_i R_{p\bar j} +S^{\bar q}_{\bar j}R_{i\bar q} \right].
\end{split}
\end{equation}
\end{lma}

We also have the following evolution equations for higher order derivative which is a sight modification of \cite[Lemma 7.1-7.2]{StreetTian2011}. 
\begin{lma}\label{evo-cur-tor}
Suppose $g(t)$ is a solution to the Hermitian Ricci flow, then the Chern-curvature and the torsion of $g(t)$ satisfy the following equations.
\begin{equation}
\begin{split}
\frac{\partial}{\partial t} \nabla^k Rm
&=\Delta \nabla^k Rm+\sum_{j=0}^k \nabla^j T * \nabla^{k+1-j} Rm+\sum_{j=0}^k \nabla^j Rm*\nabla^{k-j}Rm\\
&\quad +\sum_{j=0}^k\sum_{l=0}^j \nabla^l T *\nabla^{j-l}T*\nabla^{k-j} Rm,\\
\frac{\partial}{\partial t}\nabla^k T&=\Delta \nabla^k T+\sum_{j=0}^{k+1}\nabla^jT*\nabla^{k+1-j}T +\sum_{j=0}^k \nabla^j T*\nabla^{k-j}Rm.
\end{split}
\end{equation}
\end{lma}
\begin{proof}
The proof is identical to that in \cite[Section 7]{StreetTian2011} except now $Q\equiv 0$ and hence the last term in their formula vanishes.
\end{proof}

\section{a priori estimates}\label{Shi-type}
In this section, we will establish some local estimates for the Hermitian Ricci flow $g(t)$ on compact subset. We first need some estimates on distance function.

\begin{lma}\label{dist-esti}
Suppose $(M^n,g)$ be a complete noncompact Hermitian manifold with complex dimension $n$ and $\mathrm{BK}_{g_0}\geq -K$ on $B_{g_0}(p,2r)$. Let $p\in M$ and   $d_{g}(x,p)$ be the distance from $p$ with respect to $g$, then whenever $d_{g}(x,p)\in [\frac{1}{\sqrt{K}},r]$,
$$\partial_i \partial_{\bar j} d_{g}(x,p) \leq C_n\sqrt{K}g_{i\bar j}$$
within the cut-locus of $p$.
\end{lma}
\begin{proof}
Consider $\tilde g=Kg$, then $|Rm(\tilde g)|+|\tilde T|^2 \leq 1$ on $B_{\tilde g}(p,2\sqrt{K} r)$. We may apply \cite[Theorem 1.1]{Yu2018}, although it is stated globally, one can check easily that the proof only requires bisectional curvature lower bound locally. Therefore, we have on $B_{\tilde g}(p,\sqrt{K}r) \setminus B_{\tilde g}(p,1)$.
$$\partial_i \partial_{\bar j} d_{\tilde g}(x,p) \leq C_n\sqrt{K}\tilde g_{i\bar j}.$$
The result follows after we rescale it back to $g$.
\end{proof}

\begin{prop}\label{first-esti}
There is $\e_n>0$ such that the following holds. Suppose $g(t)$ is a solution to \eqref{HRF} on $B_{g_0}(p,r+\delta)\times [0,\tau]$ for some $p\in M$, $r,\delta>0$. If the Hermitian Ricci flow solution $g(t)$ satisfies 
\begin{align}\label{zero-cond}
(1+\e_n)^{-1} g_0\leq g(t)\leq (1+\e_n) g_0
\end{align}
on $B_{g_0}(p,\delta+r)\times [0,\tau]$. Let $K=\sup_{B_{g_0}(p,2r)}\sum_{i=0}^1|\nabla^i_{g_0} Rm(g_0)|+|\nabla^i_{g_0} T_{g_0}|$, then there is $C(n,\delta,K)>0$ such that on $B_{g_0}(p,r)\times [0,\tau]$, 
$$|Rm|+|\nabla T|+|\nabla_{g_0} g(t)|\leq C(n,\delta,K).$$
\end{prop}

\begin{proof}
In what follows, we will use $C_i$ to denote any generic constant depending only on $n,r,K$. For notational convenience, we will use  $\hat R$, $\hat T$ to denote the geometric quantities of $g_0$.

We first show the bound on $|\nabla_{g_0} g(t)|$. By our assumption and Lemma \ref{evo-trace-1}, the function $\Lambda=tr_gg_0$ satisfies 
\begin{align}
\heat \Lambda &\leq -\frac{1}{1+\e_n} |\Psi|^2+ C_0.
\end{align}
On the other hand, since we have 
\begin{equation}
\begin{split}
R_{i\bar jk}\,^l&=\hat R_{i\bar jk}\,^l+ \partial_{\bar j} \Psi_{ik}^l;\\
T_{ij}^k&=\hat T_{ij}^k+\Psi_{ij}^k -\Psi_{ji}^k.
\end{split}
\end{equation}

Therefore, together with  Lemma \ref{firstorder-diff}, the function $|\Psi|^2$ where $\Psi=\Gamma_{g_0}-\Gamma_{g(t)}$ satisfies 
\begin{equation}
\begin{split}\label{useful-evo}
\heat |\Psi|^2 &\leq -|\nabla \Psi|^2-|\bar\nabla \Psi|^2\\
&\quad +2Re\left[ g^{p\bar q}g^{i\bar j} g^{k\bar l}g_{r\bar s}\Psi^{\bar s}_{\bar j\bar l}(T^a_{pi}R_{a\bar qk}\,^r+\nabla_p\hat R_{i\bar qk}\,^r)\right]\\
&\leq -\frac{1}{2}\left( |\nabla \Psi|^2+|\bar\nabla \Psi|^2\right)+8|\Psi|^4+C_1
\end{split}
\end{equation}

Let $G(s)=e^{As}+B$ where $A$ and $B$ are some positive constants to be specified. Let $\phi(s)$ be a cutoff function on $[0,+\infty)$ such that $\phi\equiv 1$ on $[0,r+\delta/2]$, vanishes outside $[0,r+\delta]$ and satisfies 
$$ |\phi'|^2\leq 100\delta^{-2}\phi , \; \phi''\geq -100\phi\delta^{-2}.$$
Define $\Phi(x)=\phi(d_{g_0}(x,p))$ to be a cutoff function on $M$ where $d_{g_0}(x,p)$ is the distance from the fixed point $p$ using Hermitian metric $g_0$. By Lemma \ref{dist-esti} and the trick of Calabi, we may assume $d_{g_0}(x,p)$ to be smooth and satisfy $\ddb d_0(x,p) \leq C_3(n,K)\omega_{g_0}$ when we apply maximum principle.

Consider the function $F(x,t)=|\Psi|^2 G(\Lambda)$ on $M\times [0,\tau]$. Then on $B_{g_0}(p,r+\delta)\times [0,\tau]$, it satisfies
\begin{equation}
\begin{split}
 \heat F
&=G  \heat |\Psi|^2+ |\Psi|^2  \heat G\\
&\quad 
-2{\bf Re} \left(\Phi g^{i\bar j}G_{\bar j} \cdot \partial_i |\Psi|^2 \right)\\
&\leq -\frac{1}{2}\left(|\nabla \Psi|^2+|\bar \nabla \Psi|^2\right) (e^{A\Lambda}+B)\\
&\quad +|\Psi|^4\left[ 8(e^{A\Lambda}+B)-\frac{A}{1+\e_n}e^{A\Lambda}\right] \\
&\quad -A|\Psi|^2|\nabla \Lambda|^2 e^{A\Lambda}+2e^{A\Lambda} A |\nabla\Lambda||\Psi||\nabla\Psi|+C_4.
\end{split}
\end{equation}
If we choose $B=10e^{An(1+\e_n)}$, then we can use Cauchy inequality to simplify it as 
\begin{equation}
\begin{split}
 \heat F&\leq -\frac{1}{4}\left(|\nabla \Psi|^2+|\bar \nabla \Psi|^2\right) (e^{A\Lambda}+B)+C_5\\
 &\quad +\frac{1}{1+\e_n}|\Psi|^4\left[90e^{An(1+\e_n)}-Ae^{An(1-\e_n)}\right].
\end{split}
\end{equation}
Hence, we can choose $A$ sufficiently large such that 
$$100e^{An(1+\e_n)}-Ae^{An(1-\e_n)}<0$$
provided that $\e_n$ is small enough. With this choice of $A,B$ and $\e_n$, $F$ satisfies 
\begin{equation}\label{equ-F}
\begin{split}
 \heat F&\leq -c_6 F^2+C_7.
\end{split}
\end{equation}

Using \eqref{equ-F}, we can apply maximum principle on function $F\cdot \Phi$. If the maximum is attained at $t=0$, then the conclusion is trivially true. Suppose it is attained at $t=t_0>0$, then 
\begin{equation}\label{cut-afa}
\begin{split}
0&\leq \heat (F\Phi) \\
&\leq \Phi  \heat F+ F\heat \Phi -2{\bf Re}\left(g^{i\bar j} F_i \Phi_{\bar j} \right)\\
&\leq -c_6F^2\Phi +C_8 + F\left( C_8 \delta^{-2}+C_8\right)
\end{split}
\end{equation}
which implies $F\Phi$ is bounded above by $C_9(r^{-4}+1)$ at its maximum point. In particular, on $B_{g_0}(p,r+\delta/2)\times [0,\tau]$, 
$$|\Psi|^2 \leq C_{10}(\delta^{-4}+1).$$
This shows the bound on $|\nabla_{g_0} g(t)|$. 

For $|Rm|$ and $|\nabla T|$, the proof is similar but simpler. By Lemma \ref{evo-cur-tor} with Cauchy inequality, the function $G=\sqrt{|Rm|^2+|\nabla T|^2}$ satisfies 
\begin{align}
\heat G&\leq C_nG^{2}+C_n
\end{align}
whenever $|Rm|^2+|\nabla T|^2 \neq 0$. 

On the other hand, since we have established the estimate on $\Psi$, \eqref{useful-evo} have the following form now. 
\begin{equation}
\begin{split}
\heat |\Psi|^2 &\leq -\frac{1}{2}\left(|\nabla \Psi|^2 +|\bar\nabla\Psi|^2\right)+C_9
\end{split}
\end{equation}
on $B_{g_0}(p,r+\delta/2)\times [0,\tau]$. Because 
\begin{equation}
\begin{split}
\bar\nabla\Psi&=Rm(g(t))-Rm(g_0);\\
\nabla T&=\nabla (T-T_{g_0})+\Psi* T_{g_0}=\nabla \Psi+\Psi* T_{g_0},
\end{split}
\end{equation}
we can rewrite \eqref{useful-evo} to be 
\begin{equation}
\begin{split}
\heat |\Psi|^2 &\leq -c_{10}G^2+C_{11}.
\end{split}
\end{equation}
Then for $L(n,K,\delta)$ sufficiently large, the function $H=G+L|\Psi|^2$ satisfies 
$$\heat H\leq -H^2+C_{12}.$$
Therefore, we may use cutoff function trick again as \eqref{cut-afa} to show the bound on $B_{g_0}(p,r)\times [0,\tau]$.
\end{proof}

\begin{prop}\label{higher-est}
Suppose $g(t)$ is a solution to \eqref{HRF} on $B_{g_0}(p,r+\delta)\times [0,\tau]$ for some $p\in M$, $r,\delta>0$. If the Hermitian Ricci flow solution $g(t)$ satisfies 
\begin{align}
|Rm(g(t))|+|T|^2\leq K_0
\end{align}
on $B_{g_0}(p,r+\delta)\times [0,\tau]$ for some $K_0>0$. Let $B_m$ be such that 
$$\sup_{B_{g_0}(p,r+\delta)}\sum_{i=0}^m|\nabla^i Rm(g_0)|+|\nabla^{i}_{g_0} T(g_0)|\leq B_m.$$

Then for any $m\in \mathbb{N}$, there is $C(n,m,\delta,B_m,K_0)>0$ such that on $B_{g_0}(p,r+\frac{\delta}{m+1})\times [0,\tau]$, 
\begin{equation}
\begin{split}
|\nabla^m Rm|+|\nabla^{m} T|&\leq {C}.
\end{split}
\end{equation}
\end{prop}
\begin{proof}
The proof is similar to that in Proposition \ref{Improved-est}. We prove the assertion by induction on $m$. In the proof, we will denote any generic constant depending only on $n,m,\delta,B_m,K_0$ by $C_i$. Assumption implies that the conclusion is true for $m=0$. Assume it is true for $m=0,1,...,k-1$ for some $k\in \mathbb{N}$. By Lemma \ref{evo-cur-tor} and the induction hypothesis, the function $H_i=|\nabla^i Rm|^2+|\nabla^{i}T|^2$ satisfies 
\begin{equation}
\begin{split}
\heat H_{k-1}&\leq -\frac{1}{2} H_k +C_1\\
\heat H_k &\leq -\frac{1}{2}H_{k+1} +C_1H_k +C_1.
\end{split}
\end{equation}
Define the function $H=H_k (H_{k-1}+A)$ where $A$ is some large constant to be specified later. Then 
\begin{equation}
\begin{split}
\heat H&\leq  H_k \heat H_{k-1}+ (A+H_{k-1}) \heat H_k \\
&\quad -2{\bf Re}\left(g^{i\bar j} \partial_i H_k\cdot  \partial_{\bar j} H_{k-1} \right)\\
&\leq -\frac{1}{2} H_k^2+C_2H_k  +C_2H_k \sqrt{H_{k+1}H_{k-1}}\\
&\quad +(A+H_{k-1})\left[ -\frac{1}{2}H_{k+1}+C_1 +C_1H_k \right]
\end{split}
\end{equation}
where we have used the fact that $|\nabla H_i|\leq  \sqrt{H_i H_{i+1}}$. By Cauchy inequality again, if $A$ is sufficiently large, then
\begin{equation}
\begin{split}
\heat H&\leq H_k^2 \left( -\frac{1}{2}+\frac{C_n}{A}\right) +C_3 H_k +C_3\\
&\leq -\frac{1}{4}H_k^2 +C_3 H_k  +C_3\\
&\leq -c_4H^2 +C_5
\end{split}
\end{equation}
on $B_{g_0}(p,r+\delta/k)\times [0,\tau]$. Now the evolution equation is in the standard form. Let $d_{g_0}(x,p)$ and $\Phi=\phi(d_{g_0}(x,p))$ where $\phi$ is a cutoff function on $[0,+\infty)$ such that $\phi\equiv 1$ on $[0,r+\frac{\delta}{k+1}]$, vanishes outside $[0,r+\frac{\delta}{k}]$ and satisfies $|\phi'|^2\leq C_6\phi$ and $\phi''\geq -C_6\phi$. By our assumption, $g(t)$ is uniformly equivalent to $g_0$. Together with Lemma \ref{dist-esti}, we conclude that if the function $H \cdot \Phi$  achieves its maximum at $(x_0,t_0)$ where $t_0>0$, then 
\begin{equation}
\begin{split}
\heat (H\Phi)&\leq -c_4 H^2\Phi +C_6H + C_6.
\end{split}
\end{equation}
By maximum principle, $H_k \leq C_7$ on $B_{g_0}(p,r+\frac{\delta}{k})\times [0,\tau]$.
\end{proof}

In fact, the higher order derivatives of $Rm$ and $T$ will be instantly bounded after $g(t)$ evolves. 
\begin{prop}\label{Improved-est}
Suppose $g(t)$ is a solution to \eqref{HRF} on $B_{g_0}(p,r+\delta)\times [0,\tau]$ for some $p\in M$, $r,\delta>0$. If the Hermitian Ricci flow solution $g(t)$ satisfies 
\begin{align}\label{assump-cur}
|Rm|+|T|^2\leq K_0
\end{align}
on $B_{g_0}(p,r+\delta)\times [0,\tau]$ for some $K_0>0$. Then for any $m\in \mathbb{N}$, there is $C_0(n,m,\tau,\delta,K_0)>0$ such that on $B_{g_0}(p,r+\frac{\delta}{m+1})\times [0,\tau]$, 
\begin{equation}
\begin{split}
|\nabla^m Rm|+|\nabla^{m} T|&\leq {C_0}t^{-m/2}.
\end{split}
\end{equation}
\end{prop}
\begin{proof}
In what follows, we will use $C_i$ to denote any generic constants depending only on $n,m,S,\delta,K_0$. We prove the assertion by induction on $m$. Assumption ensures that it is true when $m=0$. Assume it is true when $m=0,1,...,k-1$ for some $k\in \mathbb{N}$. Let $G_i=t^i H_i$ where $H_i$ is defined in the exactly same way as in the proof of Proposition \ref{higher-est}. By Lemma \ref{evo-cur-tor} and the induction hypothesis, 
\begin{equation}
\begin{split}
\heat G_{k-1} &\leq -\frac{1}{2}G_k t^{-1} +C_1 t^{-1}\\
\heat  G_k&\leq -\frac{1}{2}G_{k+1}t^{-1}+C_1 +C_1G_k t^{-1} .
\end{split}
\end{equation}
Consider the new function $G=G_k(A+G_{k-1})$. Argue as in the proof of Proposition \ref{higher-order-est}
\begin{equation}
\begin{split}
\heat G&\leq G_k^2t^{-1} \left( -\frac{1}{2}+\frac{C_n}{A}\right) +C_3 G_k t^{-1} +C_3\\
&\leq -\frac{1}{4}G_k^2t^{-1} +C_3 G_k t^{-1} +C_3\\
&\leq -c_4G^2 t^{-1}+C_5 G t^{-1} +C_5
\end{split}
\end{equation}
on $B_{g_0}(p,r+\delta/k)$ provided that $A$ is sufficiently large.

Let $\phi(s)$ be a cutoff function on $[0,+\infty)$ such that $\phi\equiv 1$ on $[0,r+\frac{\delta}{k+1}]$, vanishes outside $[0,r+\delta/k]$ and satisfies 
$$ |\phi'|^2\leq C_6\phi , \; \phi''\geq -C_6\phi.$$
Define $\Phi(x)=\phi(d_{g_0}(x,p))$ to be a cutoff function on $M$ where $d_{g_0}(x,p)$ is the distance from the fixed point $p$ using the Hermitian metric $g_0$. By \eqref{assump-cur}, flow equation \eqref{HRF} and Lemma \ref{dist-esti}, if the function $F=\Phi G$ attains its maximum at $(x_0,t_0)$ where $t_0>0$, then at this point
\begin{equation}
\begin{split}
0&\leq \heat F\\
&\leq \Phi \heat G+ G\heat \Phi +2G\frac{|\nabla\Phi|^2}{\Phi}\\
&=-c_4 F^2 \Phi^{-1} t^{-1} +C_7F \Phi^{-1} t^{-1}+C_7.
\end{split}
\end{equation}
Hence, $F$ is bounded from above by some constant $C_7$ at this point and hence on $M\times [0,\tau]$. If $t_0=0$, then the conclusion trivially holds. Hence, the statement is true for $m=k$. By induction, this completes the proof.
\end{proof}

{ 
\section{Short-time existence under bounded geometry}
In this section, we consider the short time existence to \eqref{HRF} on complete noncompact Hermitian manifolds with bounded geometry of infinity order. Let us first recall the definition of bounded geometry:
\begin{defn}\label{boundedgeom} Let $(M^n, g)$ be a complete Hermitian manifold. Let  $k\ge 1$ be an integer and $0<\alpha<1$. $g$ is said to have   bounded geometry of   order $k+\alpha$ if there are positive numbers $r, \kappa_1, \kappa_2$  such that at every $p\in M$ there is a neighborhood $U_p$ of $p$, and local biholomorphism $\xi_{p}$ from $D(r)$ onto $U_p$ with $\xi_p(0)=p$ satisfying  the following properties:
  \begin{itemize}
    \item [(i)] the pull back metric $\xi_p^*(g)$   satisfies:
    $$
    \kappa_1 g_e\le \xi_p^*(g)\le \kappa_2 g_e$$ where $g_e$ is the standard metric on $\C^n$;
    \item [(ii)] the components $g_{\ijb}$ of $\xi_p^*(g)$ in the natural coordinate of $D(r)\subset \C^n$  are uniformly bounded in the standard $C^{k+\alpha}$ norm in $D(r)$ independent of $p$.
  \end{itemize}
$(M,g)$ is said to have bounded geometry of infinity order if instead of (ii) we have for any $k$, the $k$-th derivatives of $g_\ijb$ in $D(r)$ are bounded by a constant independent of $p$. $g$ is said to have bounded geometry of infinite order on a compact set $\Omega$ if (i) and (ii) are true for all $k$ for all $p\in \Omega$.
 \end{defn}

From Lemma \ref{para-sys}, we see that the Hermitian Ricci flow equation is strongly parabolic if $g(t)$ is uniformly equivalent to some fixed metric, say for example $g_0$. Moreover, we can freely replace the Chern connection in Lemma \ref{para-sys} by the Levi-Civita connection since $g_0$ is assumed to have bounded geoemtry of infinity order. The short-time existence result will then follow by a standard inverse function theorem argument. For more details, we refer readers  to \cite[section 3-4]{Shi1989}, \cite[ Chapter VII, Theorem 7.1]{LSU1968}, \cite[Section 4]{SSS1} and \cite[Theorem  3.7.1]{Bamler2011}.
\begin{thm}\label{short-time-1}
Let $(M,g_0)$ be a complete noncompact Hermitian manifold with bounded geometry of infinity order. Then there is $\tau(n,g_0)>0$ such that \eqref{HRF} has a solution on $M\times [0,\tau]$. Moreover,  on $M\times [0,\tau]$, we have 
$$(1+\e_n)^{-1}g_0\leq g(t)\leq (1+\e_n) g_0$$
where $\e_n$ is the constant in Proposition \ref{first-esti}.
\end{thm}
\begin{proof}
Since $g_0$ has bounded geomtry of infinity order, we are free to interchange the Levi-Civita connection of $g_0$ and the Chern connection of $g_0$. Therefore by Lemma \ref{para-sys}, the equation of the Hermitian Ricci flow has the form 
$$(\partial_t-g^{AB}\tilde\nabla^L_A\tilde\nabla^L_B) g(t)=Q(g,\tilde\nabla^L g)$$
where $\tilde\nabla^L$ denotes the Levi-civita connection of $g_0$. Hence, it has exactly same form as the Deturck Ricci flow. Since we assume $g_0$ to have bounded geometry of infinity order,   the  proof of \cite[Theorem  3.7.1]{Bamler2011} can be carried over (the first case in the proof). It is clear from the argument in \cite[Theorem  3.7.1]{Bamler2011} that $g(t)$ can be as close to $g_0$ as we wish by  shrinking the existence time.
\end{proof}

\section{General Short-time existence on $M$}
In this section, we will show that one can construct a solution to \eqref{HRF} with uniformly bounded $|Rm|+|T|^2$ if $g_0$ has bounded $|Rm(g_0)|+|T_{g_0}|^2+|Rm^L(g_0)|$. In the celebrated work by Shi \cite{Shi1989}, Shi showed that in fact the constructed solution of the DeTurck Ricci flow will have bounded curvature by establishing an integral estimate. And therefore, the curvature of the corresponding Ricci flow is uniformly bounded for a short time. For the Hermitian Ricci flow, the integral estimate is a bit tedious due to the presence of torsion. To bypass the complicated integration argument, we take an alternative path using the idea in \cite{LeeTam2017}. By the work of \cite{Shi1989,Tam2010}, there is an exhaustion function $\rho$ with $|\partial \rho|^2+|\ddb \rho|\leq C$ if $g_0$ has  bounded Riemannian curvature because the Levi-Civita connection and the Chern connection only differ by the torsion (see for example \cite{ZhengYang2016}). In this section, we will assume $(M,g_0)$ to be a complete noncompact Hermitian manifold satisfying the followings.
\begin{enumerate}
\item[\bf (A)] $\sup_M |Rm(g_0)|+|T_{g_0}|^2 \leq K_0$;
\item[\bf (B)] There is an exhaustion function $\rho\geq 1$ such that 
$$\sup_M |\partial\rho|^2+|\ddb \rho|\leq K_0.$$
\end{enumerate}

We will proceed as in \cite{LeeTam2017}. Let $\kappa\in (0,1)$, $f:[0,1)\to[0,\infty)$ be the function:
\be\label{e-exh-1}
 f(s)=\left\{
  \begin{array}{ll}
    0, & \hbox{$s\in[0,1-\kappa]$;} \\
    -\displaystyle{\log \lf[1-\lf(\frac{ s-1+\kappa}{\kappa}\ri)^2\ri]}, & \hbox{$s\in (1-\kappa,1)$.}
  \end{array}
\right.
\ee
Let   $\varphi\ge0$ be a smooth function on $\R$ such that $\varphi(s)=0$ if $s\le 1-\kappa+\kappa^2 $, $\varphi(s)=1$ for $s\ge 1-\kappa+2 \kappa^2 $
\be\label{e-exh-2}
 \varphi(s)=\left\{
  \begin{array}{ll}
    0, & \hbox{$s\in[0,1-\kappa+\kappa^2]$;} \\
    1, & \hbox{$s\in (1-\kappa+2\kappa^2,1)$.}
  \end{array}
\right.
\ee
such that $\displaystyle{\frac2{ \kappa^2}}\ge\varphi'\ge0$. Define
 $$\mathfrak{F}(s):=\int_0^s\varphi(\tau)f'(\tau)d\tau.$$
 
Here we collect some useful lemmas from \cite{LeeTam2017}.
\begin{lma}\label{l-exhaustion-1} Suppose   $0<\kappa<\frac18$. Then the function $\mathfrak{F}\ge0$ defined above is smooth and satisfies the following:
\begin{enumerate}
  \item [(i)] $\mathfrak{F}(s)=0$ for $0\le s\le 1-\kappa+\kappa^2$.
  \item [(ii)] $\mathfrak{F}'\ge0$ and for any $k\ge 1$, $\exp( -k\mathfrak{F})\mathfrak{F}^{(k)}$ is uniformly  bounded.
  \item [(iii)]  For any $ 1-2\kappa <s<1$, there is $\tau>0$ with $0<s -\tau<s +\tau<1$ such that
 \bee
 1\le \exp(\mathfrak{F}(s+\tau)-\mathfrak{F}(s-\tau))\le (1+c_2\kappa);\ \ \tau\exp(\mathfrak{F}(s_0-\tau))\ge c_3\kappa^2
 \eee
  for some absolute constants  $c_2>0, c_3>0$.
\end{enumerate}
\end{lma}

For   any  $\rho_0>0$, let  $U_{\rho_0}$ be the component of $ \{x|\ \rho(x)<\rho_0\}$ containing a fixed point $p\in M$. Hence $U_{\rho_0}$ will exhaust $M$ as $\rho_0\to\infty$. For $\rho_0>>1$, let $F(x)=\mathfrak{F}(\rho(x)/\rho_0)$. Let $h_{\rho_0}=e^{2F}g_0$. Then $(U_{\rho_0},h_{\rho_0})$ is a complete Hermitian metric, see \cite{Hochard2016}, and $h_{\rho_0}=g_0$ if on $\{\rho(x)<(1-\kappa+\kappa^2)\rho_0\}$.
\begin{lma}\label{l-boundedgeom}
$(U_{\rho_0},h_{\rho_0})$ has bounded geometry of infinite order.
\end{lma}
\begin{proof}
This is Lemma 4.3 in \cite{LeeTam2017}.
\end{proof}

Moreover, under the assumption {\bf (A)} and {\bf (B)}, we have 
\begin{lma}\label{l-bounded-cur}
For $\rho_0$ sufficiently large, we have 
$$\sup_{U_{\rho_0}} |Rm(h_{\rho_0})|+|T(h_{\rho_0})|^2\leq 2K_0.$$
\end{lma}
\begin{proof}
This follows directly from \cite[Appendix B]{LeeTam2017} and Lemma \ref{l-exhaustion-1}.
\end{proof}

Now we are ready to get the short-time existence for the Hermitian Ricci flow under assumption described above which covers Theorem \ref{MAIN-EX}.
\begin{thm}\label{Improved-short-time}
Suppose $(M^n,g_0)$ is a complete noncompact Hermitian manifold with complex dimension $n$ so that  {\bf (A)} and {\bf (B)}  hold
for some $K_0>0$. 
then there is a short-time solution to \eqref{HRF} with initial metric $g(0)=g_0$ on $M\times [0,c_nK_0^{-1}]$ which satisfies 
\begin{align}\label{doubletime-1}
\sup_{M\times [0,c_nK_0^{-1}]} \left( |Rm(g)|+|T_{g}|^2\right)\leq 4K_0.
\end{align} Moreover, for all $m\in \mathbb{N}$, there is $C(n,m,K_0)>0$ so that on $M\times (0,c_nK_0^{-1}]$
$$|\nabla^{m}T_{g}|^2+|\nabla^m Rm(g)|^2\leq \frac{C(n,m,K_0)}{t^m}.$$
\end{thm}
\begin{proof}
Let $(U_{\rho_i},g_{0,i})$ be the sequence of Hermitian metric constructed using above method. By Lemma \ref{l-boundedgeom} and Theorem \ref{short-time-1}, there is a short-time solution $g_i(t)$ to \eqref{HRF} on each $U_{\rho_i}$ with initial metric $g_{0,i}=h_{\rho_i}$. Let $\tau_i$ be the maximal time such that 
$$(1+\e_n)^{-1}g_{0,i}\leq g_i(t) \leq (1+\e_n)g_{0,i}\quad\text{on}\;\;U_{\rho_i}\times [0,\tau_i]$$
where $\e_n$ is the constant from Proposition \ref{first-esti}. By Proposition \ref{first-esti},  $g_i(t)$ satisfies 
\begin{align}\label{boundedness}\sup_{U_{\rho_i}\times [0,\tau_i]} \max\left\{ |Rm(g_{i})|,|T_{g_{i}}|^2\right\}<+\infty\end{align}
By Lemma \ref{evo-cur-tor} and \eqref{boundedness}, the function $F=|T|^4+|Rm|^2$ is bounded and satisfies 
$$\heat F\leq c_nF^\frac{3}{2}.$$

Therefore, we may apply maximum principle (see for example \cite[Lemma 3.4]{HuangLeeTamTong2018})  to conclude that on $[0,\tau_i]\cap [0,c_nK_0^{-1}]$, 
\begin{align}\label{cur-exh-time}
\sup_{U_{\rho_i}} \left\{ |Rm(g_{i})|+|T_{g_{i}}|^2\right\}<4K_0.
\end{align}

\begin{claim}
There is $c_n>0$ such that $\tau_i\geq c_nK_0^{-1}$ for all $i\in\mathbb{N}$.
\end{claim}
\begin{proof}[Proof of Claim.] Suppose $\tau_i<c_nK_0^{-1}$. Since $$(1+\e_n)^{-1} g_{0,i}\leq g_i(t)\leq (1+\e_n) g_{0,i}$$
on $U_{\rho_i}\times [0,\tau_i)$.  By the above discussion, if $c_n$ is sufficiently small, then \eqref{cur-exh-time} holds on $[0,\tau_i]$.

By Proposition \ref{first-esti} and Proposition \ref{higher-est}, for any $m\in\mathbb{N}$, there is $C(n,m,U_{\rho_i})>0$ such that on $U_{\rho_i}\times [0,\tau_i)$, 
$$|\nabla^m Rm\left(g_i(t)\right)|\leq C(n,m,U_{\rho_i}).$$

Denote $\tilde \nabla=\nabla_{g_{0,i}}$. When $m=1$, since $\partial_t (\tilde \nabla g)=(\tilde\nabla -\nabla) S+ \nabla S$
\begin{equation}
\begin{split}
\partial_t |\tilde \nabla g|^2\leq C_1+C_1|\tilde \nabla g|^2.
\end{split}
\end{equation}
Hence, $|\nabla_{g_{0,i}} g_i(t)|\leq C(n,U_{\rho_i})$ on $[0,\tau_i)$. Inductively, we can show that for any $m\in\mathbb{N}$, there is $C(n,m,U_{\rho_i})$ such that on $U_{\rho_i}\times [0,\tau_i)$, 
$$|\nabla^m_{g_{0,i}} g_i(t)|\leq C(n,m,U_{\rho_i}).$$

Therefore, we may take subsequent limit on $g_i(t)$, $t\rightarrow \tau_i$ to obtain $g_i(\tau_i)$ which has bounded geometry of infinity order. By Theorem \ref{short-time-1}, $g_i(t)$ exists on $[0,\tau_i+\e)$ for some $\e>0$. Moreover, if $c_n$ is small enough, then \eqref{cur-exh-time} implies that
$$(1+\e_n)^{-1} g_{0,i}\leq g_i(t)\leq (1+\e_n) g_{0,i}$$
holds on $[0,\tau_i+\e)$ which contradicts with the maximality.
\end{proof}

By \eqref{cur-exh-time}, Proposition \ref{higher-est} and flow equation \eqref{HRF}, we can use similar argument as above to show that on any compact set $\Omega$ and any $m\in \mathbb{N}$, there is $C(m,n,K_0,\Omega)>0$ such that for any $i>>1$, we have on 
$$\sup_{\Omega \times [0,c_nK_0^{-1}]}|\nabla^m_{g_0}g_i(t)| \leq C(m,n,K_0,\Omega).$$

Hence, we may take a subsequence $i_k\rightarrow \infty$ to obtain a limiting solution $g(t)$ on $M\times [0,c_nK_0^{-1}]$ with $g(0)=g_0$ and 
$$\sup_{M\times [0,c_nK_0^{-1}]} \left\{ |Rm(g)|+|T_{g}|^2\right\}\leq 4K_0.$$
The higher order derivatives follows from Proposition \ref{Improved-est}. 
\end{proof}

\section{Hermitian Ricci flow on nonpositively curved manifolds}\label{pre}
In this section, we will apply the Hermitian flow to study complete noncompact Hermitian manifolds with non-positive bisectional curvature. In particular, we will 
generalize the preservation of non-positive Chern-Ricci curvature in \cite{Lee2018} to complete noncompact case. We will first prove the following. 

\begin{thm}\label{preserve-Ric}
Suppose $(M,g(t))$ is a complete noncompact solution to \eqref{HRF} on $M\times [0,\tau]$ with 
\begin{align}\label{doubletime}
\sup_{M\times [0,\tau]}\left( |Rm|+|T|^2\right)\leq K
\end{align}
for some $K>1$. If $g(0)=g_0$ has non-positive bisectional curvature, then there is $c_1(n),c_2(n)>0$ such that for all $(x,t)\in M\times [0,\tau]\cap [0,c_1K^{-1}]$,
\begin{enumerate}
\item $Ric_t\leq 0$;
\item $|R_{u\bar vx\bar x}|^2\leq (20+c_2\sqrt{Kt})|g_{x\bar x}|^2|R_{u\bar u}||R_{v\bar v}|$ for all $x,u,v\in T^{1,0}M$.
\end{enumerate}
\end{thm}

In fact, the curvature preservation conditions were first considered in \cite{BohmWilking2007} where they considered Riemannian manifolds with nonnegative sectional curvature. We would like to point out that to establish weak maximum principle on curvature conditions along noncompact flow with bounded curvature, usually one will consider $R^\e_{i\bar jk\bar l}=R_{i\bar jk\bar l}-\e \rho B_{i\bar jk\bar l}$ where $\rho$ is a distance function from some fixed point (see for example \cite[Chapter 12]{ChowRicciflow2}) so that one can localize the argument on compact set. By showing that $R^\e$ is "$\e$-close" to the desired curvature conditions, one can show that $R$ satisfies the goal by letting $\e\rightarrow 0$. However, since the second curvature condition in Theorem \ref{preserve-Ric} does not explicitly satisfy the null vector condition (see for example \cite[Theorem 12.33]{ChowRicciflow2}, this approach fails due to the presence of a quadratic term. We here take an alternative approach relying on parabolic rescaling argument. We first prove the following weaker version. 
\begin{prop}\label{al-pre-Ric}
Under the assumption in Theorem \ref{preserve-Ric}, there is $c_1(n),c_2(n)>0$ such that for all $(x,t)\in M\times [0,\tau]\cap [0,c_1K^{-1}]$, the curvature type tensor $\hat R_{i\bar jk\bar l}=R_{i\bar jk\bar l}-B_{i\bar jk\bar l}$ satisfies \begin{enumerate}
\item $Ric(\hat R)\leq -e^{-c_2Kt}g$;
\item $|\hat R_{u\bar vx\bar x}|^2\leq (20+c_2\sqrt{Kt})|g_{x\bar x}|^2|\hat R_{u\bar u}||\hat R_{v\bar v}|$ for all $x,u,v\in T^{1,0}M$.
\end{enumerate}
\end{prop}

We first prove Theorem \ref{preserve-Ric} by assuming the Proposition \ref{al-pre-Ric}. 

\begin{proof}[Proof of Theorem \ref{preserve-Ric}]
For any $L>>1$, define $\tilde g(t)=L^{-1} g(L t)$ on $M\times [0,\tau L^{-1}]$. Then $\tilde g(0)$ has non-positive bisectional curvature and $\tilde g(t)$ satisfies 
\begin{align}
\sup_{M\times [0,\tau/L]}\left( |\widetilde{Rm}|+|\tilde T|^2\right)\leq KL.
\end{align}
Apply Proposition \ref{al-pre-Ric} on $\tilde g(t)$ and then rescale it back to $g(t)$, we have for $t\in [0,\tau]\cap [0,c_1K^{-1}]$, $g(t)=L\tilde g(t/L)$ satisfies
\begin{enumerate}
\item $Ric(g(t))\leq L^{-1}(n+1-e^{-c_2Kt})g(t)$;
\item  for all $x,u,v\in T^{1,0}M$, 
\begin{equation*}
\begin{split}
|R_{u\bar vx\bar x}-L^{-1}B_{u\bar vx\bar x}|^2&\leq (20+c_2\sqrt{Kt})|g_{x\bar x}|^2| R_{u\bar u}-L^{-1}(n+1)g_{u\bar u}|\\
&\quad \times| R_{v\bar v}-L^{-1}(n+1)g_{v\bar v}|.
\end{split}
\end{equation*} 
\end{enumerate}
Since this is true for all $L>>1$, the conclusion follows by letting $L\rightarrow \infty$.
\end{proof}

\begin{proof}[Proof of Proposition \ref{al-pre-Ric}]
By \eqref{doubletime}, we may assume 
\begin{align}\label{metric-equ}
\frac{1}{2}g_0\leq g(t)\leq 2g_0
\end{align}
on $M\times [0,\tau]\cap [0,c_nK^{-1}]$. Let $z_0\in M$ and $d_{g_0}(x,z_0)$ be the distance from $z_0$ using the metric $g_0$. Let $\phi$ be a cutoff function on $[0,+\infty)$ such that $\phi\equiv 1$ on $[0,1]$, vanishes outside $[0,2]$ and satisfies 
$$|\phi'|^2\leq 100\phi,\;\;\phi''\geq -100\phi.$$
For any $r_0>>K^{10}$, let $\Phi(x,t)=\phi\left(\frac{d_0(x,z_0)}{r_0}\right)$ and define a curvature type tensor 
$$W_{i\bar jk\bar l}=\Phi R_{i\bar jk\bar l}-B_{i\bar jk\bar l}.$$
We will use $W_{i\bar j}$ to denote $W_{i\bar jk\bar l}g^{k\bar l}$ as well.
\begin{claim}There is $c_1(n),c_2(n)>0$ such that for all $r_0>>1$, $t\in [0,\tau]\cap [0,c_1K^{-1}]$, 
\begin{enumerate}
\item[\bf (a)] $W_{i\bar j}< -e^{-c_2Kt}g_{i\bar j}$;
\item[\bf (b)]$|W_{u\bar vx\bar x}|^2< (20+c_2\sqrt{Kt})|g_{x\bar x}|^2|W_{u\bar u}||W_{v\bar v}|$ for all $x,u,v\in T^{1,0}M$.
\end{enumerate}
\end{claim}
\begin{proof}
[Proof of Claim.] We take $c_1=\frac{1}{2c_2}$. We will specify the choice of $c_2$ in the proof below. The proof is similar to \cite[Lemma 4.1]{Lee2018} except that we have to take care of the cutoff function. Clearly, the claim is true at $t=0$, see \cite[Lemma 4.2]{Lee2018} for detailed computation. Due to the cutoff function $\Phi$, if the claim is false, there is $t_0\in (0,\tau]\cap (0,c_1K^{-1}]$ such that both {\bf (a)} and {\bf (b)} are true on $[0,t_0)$ and one of them fails at $t=t_0$. In particular, we have for all $z\in M$, $t\in [0,t_0]$, $y,u,v\in T_z^{1,0}M$ with $|y|=1$,
\begin{equation}\label{a1}
\begin{split}
W_{y,\bar y}&\leq -  e^{-c_2Kt};\\
|W_{u\bar v y\bar y }|^2&< (20+c_2\sqrt{Kt})W_{u\bar u}W_{v\bar v}.
\end{split}
\end{equation}

As in \cite[Page 1599]{Liu2014}, we may use polarization and \eqref{doubletime} to infer that for any $e_k,e_l\in T^{1,0}$ with unit $1$ and $e_i,e_j\in T^{1,0}$,
\begin{align}\label{a2}
|W_{i\bar jk\bar l}|^2 \leq C_n W_{i\bar i}W_{j\bar j}, \;\; |W_{i\bar jk\bar l}|^2 \leq C_nK |W_{i\bar i}|.
\end{align}
\noindent

{\bf Case 1:} Condition {\bf (a)} is  true on $[0,t_0)$ and fails at $t=t_0$. Then there is $p\in M$, $X_0\in T^{1,0}_p M$ with $|X_0|=1$ such that 
$$W_{X_0,\bar X_0}=- e^{-c_2Kt}.$$

Consider the following tensor $$A_{i\bar j}=\Phi R_{i\bar j}+\left[e^{-c_2Kt}-(n+1) \right] g_{i\bar j}=W_{i\bar j}+ e^{-c_2Kt}g_{i\bar j}$$
which satisfies $A(X_0,\bar X_0)=0$ and $A(Y,\bar Y)\leq 0$ for all $Y\in T^{1,0}_xM$, $x\in M$. We may assume $|X_0|_{g(t_0)}=1$ by rescaling.

Extend $X_0$ locally to a $T^{1,0}$ vector field around $(p,t_0)$ { such that at $(p,t_0)$,
\begin{equation}
\label{ext-1}
\begin{split}
\nabla_{\bar q} X^p&=0;\quad 
\nabla_{p}X^q =T^q_{pl}\;X^l.
\end{split}
\end{equation}
}
Locally, $X=X^i\frac{\partial}{\partial z^i}$. We will denote $\bar X=\overline{X^i}\frac{\partial}{\partial \bar z^i}=X^{\bar i}\frac{\partial}{\partial \bar z^i}$. Then $A(X,\bar X)$ defined a function locally and satisfies 
\begin{align}\label{max-1}
\Box \Big|_{(p,t_0)}A(X,\bar X) \geq 0.
\end{align}
where we denote $\heat$ by $\Box$ for notational convenience. Now we compute the evolution equation for $A(X,\bar X)$. At $(p,t_0)$, 
\begin{equation}\label{equ-A-1}
\begin{split}
\frac{\partial}{\partial t} A(X,\bar X)
&=\left(\partial_t A_{i\bar j}\right) X^iX^{\bar j} +A_{i\bar j} \left( \partial_t X^i X^{\bar j} +X^i \partial_t X^{\bar j}\right)\\
&=\left[\Phi\cdot \partial_t  R_{i\bar j}-[e^{-c_2Kt}-(n+1)]S_{i\bar j}- c_2Ke^{-c_2Kt}g_{i\bar j}\right]X^i X^{\bar j}\\
&\leq\Phi\cdot \partial_t  R_{i\bar j}\cdot X^i X^{\bar j}-\frac{1}{2}c_2K.
\end{split}
\end{equation}
provided that $c_2>>1$ is sufficiently large. Here we have used  \eqref{doubletime} and  the fact that for any $Y\in T^{1,0}_pM$,
\begin{align}\label{firstorder}
A_{X_0\bar Y}=0
\end{align}

Now we compute the $\Delta A(X,\bar X)$. We may in addition assume that at $(p,t_0)$, $g_{i\bar j}=\delta_{i\bar j}$. Using $\nabla g=0$, \eqref{ext-1} and \eqref{firstorder}, we have
\begin{equation}\label{equ-A-2}
\begin{split}
&\quad \Delta A(X,\bar X)\\
&=\frac{1}{2} g^{r\bar s} (\nabla_r \nabla_{\bar s}+\nabla_{\bar s}\nabla_r) \Big( A_{i\bar j }X^i X^{\bar j}\Big)\\
&=\Delta (\Phi R_{i\bar j})\cdot X^i X^{\bar j}+A_{i\bar j;\bar r} T^i_{ rp}X^pX^{\bar j} +A_{i\bar j;r}T^{\bar j}_{\bar r\bar q} X^iX^{\bar q}  +A_{i\bar j} T^i_{rp}T^{\bar j}_{\bar r\bar q}X^p X^{\bar q}\\
&=\Delta \Phi \cdot R_{i\bar j} X^i X^{\bar j}+\Phi \cdot \Delta R_{i\bar j}X^i X^{\bar j} + 2{\bf Re}\left( \nabla_r \Phi \cdot \nabla_{\bar r} R_{i\bar j}\cdot X^i X^{\bar j}\right)\\
&\quad +\nabla_{\bar r}\left(\Phi R_{i\bar j}\right) T^i_{ rp}X^pX^{\bar j} +\nabla_r\left( \Phi R_{i\bar j}\right)T^{\bar j}_{\bar r\bar q} X^iX^{\bar q}  +A_{i\bar j} T^i_{rp}T^{\bar j}_{\bar r\bar q}X^p X^{\bar q}\\
&=\Delta \Phi \cdot R_{i\bar j} X^i X^{\bar j}+\Phi \cdot \Delta R_{i\bar j}X^i X^{\bar j} + 2{\bf Re}\left( \nabla_r \Phi \cdot \nabla_{\bar r} (R_{X\bar X})\right)\\
&\quad +\Phi R_{i\bar j;\bar r} T^i_{ rp}X^pX^{\bar j} +\Phi  R_{i\bar j;r}T^{\bar j}_{\bar r\bar q} X^iX^{\bar q}  +A_{i\bar j} T^i_{rp}T^{\bar j}_{\bar r\bar q}X^p X^{\bar q}.
\end{split}
\end{equation}

By combining \eqref{equ-A-1} and \eqref{equ-A-2}, we have
\begin{equation}\label{eq-A}
\begin{split}
\Box A(X,\bar X)
&\leq -\frac{1}{2}c_2K+ \Phi \Box R_{i\bar j}\cdot X^i X^{\bar j} -2{\bf Re}\left( \nabla_r \Phi \cdot \nabla_{\bar r} (R_{X\bar X})\right)\\
&\quad -\Phi R_{i\bar j;\bar r} T^i_{ rp}X^pX^{\bar j} -\Phi  R_{i\bar j;r}T^{\bar j}_{\bar r\bar q} X^iX^{\bar q} \\
&\quad -\Delta \Phi \cdot R_{i\bar j}X^i X^{\bar j}-A_{i\bar j} T^i_{rp}T^{\bar j}_{\bar r\bar q}X^p X^{\bar q}\\
&=-\frac{1}{2}c_2K -\Delta \Phi \cdot R_{X\bar X}+(n+1-e^{-c_2Kt}) |T^i_{rX}|^2\\
& \quad -2{\bf Re}\left( \nabla_r \Phi \cdot \nabla_{\bar r} (R_{X\bar X})\right)+\Phi R_{X\bar Xp}\,^qR_{q}\,^p\\
&\quad-(n+1-e^{-c_2Kt})S_{X\bar X} 
\end{split}
\end{equation}
where we have used \eqref{firstorder} in the last step.

Since at $(p,t_0)$, we have $\nabla_r A_{X\bar X}=0$. Hence,
$$\nabla_{\bar r} \Phi \cdot R_{X\bar X}+\Phi \nabla_{\bar r} R_{X\bar X}=\left( (n+1)-e^{-c_2Kt}\right) g_{X\bar j}T^{\bar j}_{\bar r\bar X}.$$

On the other hand, as $\Phi R_{X\bar X} =(n+1-e^{-c_2Kt})$ at $(p,t_0)$ and $t_0\in [0,\tau]\cap [0,c_1K^{-1}]$, 
$$\Phi \geq \frac{1}{(n+1)K} \left[ n+1-e^{-c_2Kt}\right]\geq \frac{n}{(n+1)K}.$$

Hence, by using the properties of $\phi$ and combining with \eqref{doubletime}, \eqref{metric-equ} and $r_0>>K^2$, we have
\begin{equation}
\begin{split}\label{cutoff-1}
 -2{\bf Re}\left( \nabla_r \Phi \cdot \nabla_{\bar r} (R_{X\bar X})\right)
&=\frac{|\nabla_r \Phi|^2}{\Phi}R_{X\bar X} -\frac{n+1-e^{-c_2Kt}}{\Phi}g_{X\bar j}T^{\bar j}_{\bar r\bar X}\nabla_r \Phi\\
&\leq \frac{c'_n}{\sqrt{r_0}}.
\end{split}
\end{equation}

Similarly, 
\begin{equation}\label{cutoff-2}
\begin{split}
-\Delta \Phi \cdot R_{X\bar X}&\leq \frac{c_n}{\sqrt{r_0}}.
\end{split}
\end{equation}

We now combine \eqref{eq-A}, \eqref{cutoff-1}, \eqref{cutoff-2} and \eqref{doubletime} to show that if $c_2$ is sufficiently large depending only on $n$, then 
\begin{equation}
\begin{split}
\Box A(X,\bar X) &\leq -\frac{1}{3}c_2K+W_{X\bar X p\bar q}R_{q\bar p}\\
&\leq -\frac{1}{3}c_2K+C_n K |W_{X\bar X}|\\
&<0
\end{split}
\end{equation}
where we have used \eqref{a2}.  But this contradicts with \eqref{max-1}.

\noindent\\

{\bf Case 2:} Condition {\bf (b)} is fail at $t=t_0$. Then there is $p\in M$, $x_0,u_0,v_0\in T^{1,0}_pM$ with $|x_0|_{t_0}=1$ such that 
$$|W_{u_0\bar v_0 x_0\bar x_0}|^2=(20+c_2\sqrt{Kt}) W_{u_0 \bar u_0}W_{v_0\bar v_0}.$$
By rescaling, we may assume $|u_0|_{t_0}=|v_0|_{t_0}=1$. 
As in case 1, we extend $x_0,u_0,v_0$ to local vector field $X,U,V$ around $(p,t_0)$. {We extend $x_0$ so that along each geodesics $\gamma$  emanating from $p$, $\nabla_{\dot\gamma}X=0$ at $t=t_0$ and constant in $t$. On the other hand, we extend $u_0,v_0$ to $U$ and $V$ such that at $(p,t_0)$,
\begin{equation}\label{ext-2}
\begin{split}
\nabla_{\bar s} U^r=0 &, \quad \nabla_{p} U^r=T^r_{p q}U^q,\quad \Box U^r=\frac{1}{2}S^r_p U^p;\\
 \nabla_{\bar s} V^r=0& ,\quad \nabla_{p} V^r=T^r_{p q}V^q;\quad  \Box V^r=\frac{1}{2}S^r_p V^p;\\
 \nabla_{\bar s} X^r=0& ,\quad \nabla_{p} X^r=0;\quad\quad\quad   \Box X^r=0.\\
\end{split}
\end{equation}
}
Hence the function $$F(x,t)=g_{X\bar X}^{-2}|W_{U\bar V X\bar X}|^2-(20+c_2\sqrt{Kt})W_{U\bar U}W_{V\bar V}$$ attains its local maximum at $(p,t_0)$ and therefore satisfies
\begin{align}\label{max-2}
\Box F\Big|_{(p,t_0)}\geq 0.
\end{align}

We now differentiate each of them carefully. Using \eqref{ext-2} and Lemma \ref{Ricci-evo}, a similar calculation as in {\bf Case 1} yields 
\begin{equation*}
\begin{split}
\Box W_{U\bar U}
&=\Box W_{i\bar j}\cdot U^i U^{\bar j}+W_{i\bar j} U^i \Box U^{\bar j} +W_{i\bar j}\Box U^i\cdot  U^{\bar j}\\
&\quad -g^{r\bar s}W_{i\bar j} U^i_{;r} U^{\bar j}_{;\bar s}
-g^{r\bar s}W_{i\bar j;r}  U^iU^{\bar j}_{;\bar s}
-g^{r\bar s}W_{i\bar j;\bar s} U^i_{;r} U^{\bar j}\\
&=\Box \Phi \cdot R_{U\bar U}-2{\bf Re}\left(g^{r\bar s}\Phi_r \,\nabla_{\bar s}R_{i\bar j}\cdot U^i U^{\bar j} \right)+(n+1)S_{U\bar U}\\
&\quad -(n+1)\Phi S_{U\bar U}+(n+1) g_{i\bar j} g^{r\bar s} T^i_{rU}T^{\bar j}_{\bar s\bar U}\\
&\quad -g^{r\bar s}\Phi_r R_{U\bar j} T^{\bar j}_{\bar s \bar U}-g^{r\bar s} \Phi_{\bar s} R_{i\bar U} T^i_{rU}+\Phi R_{U\bar U k}\,^p R_{p}^k.
\end{split}
\end{equation*}
By Proposition \ref{Improved-est}, we have $|\nabla Rm|\leq C(n,K)t^{-1/2}$. Using the choice of $\phi$, Lemma \ref{dist-esti}, \eqref{doubletime},  and $r_0>>K$, we have 
\begin{equation}\label{evo-Ric-1}
\begin{split}
\Box W_{U\bar U}
&\leq C_nK+\frac{C_n}{\sqrt{t}}+\Phi R_{U\bar U k}\,^p R_{p}^k\\
&\leq \frac{C_n}{\sqrt{t}}-C_nKW_{U\bar U}
\end{split}
\end{equation}
where we have used \eqref{a1} and \eqref{a2} in the last inequality.
Similarly, 
\begin{equation}\label{evo-Ric-2}
\begin{split}
\Box W_{V\bar V}
& \leq \frac{C_n}{\sqrt{t}}-C_nKW_{V\bar V}.
\end{split}
\end{equation}

By combining \eqref{evo-Ric-1}, \eqref{evo-Ric-2} with \eqref{a1}, \eqref{a2}, \eqref{doubletime} and using the fact that $c_2>>1$, we arrive at the following inequality.
\begin{equation}
\begin{split}
&\quad \Box \left[(20+c_2\sqrt{Kt})W_{U\bar U}W_{V\bar V}\right]\\
&\geq -2(20+c_2\sqrt{Kt}){\bf Re}\left(g^{r\bar s}\nabla_r W_{U\bar U} \cdot \nabla_{\bar s}W_{V\bar V} \right)+\frac{c_2}{2}\sqrt{\frac{K}{t}} W_{U\bar U }W_{V\bar V}.
\end{split}
\end{equation}

\noindent\\
Now we derive the evolution equation of $|W_{U\bar VX\bar X}|^2$. Similar to the computation of $\Box W_{U\bar U}$, using \eqref{ext-2} and Lemma \ref{Rm-evo}, we have 
\begin{equation}\label{evo-Rm-1}
\begin{split}
\Box W_{U\bar VX\bar X}
&=\Box (\Phi R)_{i\bar jk\bar l}\cdot U^iV^{\bar j}X^kX^{\bar l}- \Box B_{i\bar jk\bar l}\cdot U^iV^{\bar j}X^kX^{\bar l}\\
&\quad + W_{i\bar jk\bar l}(\Box U^i) V^{\bar j}X^kX^{\bar l}+ W_{i\bar jk\bar l} U^i(\Box V^{\bar j})X^kX^{\bar l}\\
&\quad -g^{r\bar s} (\Phi R_{i\bar jk\bar l})_{;r}\,U^iV^{\bar j}_{;\bar s}X^kX^{\bar l} -g^{r\bar s}(\Phi R_{i\bar jk\bar l})_{;\bar s}\,U^i_{;r}V^{\bar j}X^kX^{\bar l} \\
&\quad  -W_{i\bar jk\bar l} U^i_{;r}V^{\bar j}_{;\bar s}X^kX^{\bar l} \\
&= \Phi g^{r\bar s}\Big[R_{U\bar V r}\,^pR_{p\bar sX\bar X}+R_{r\bar V X}\,^p R_{U\bar s p\bar X} -R_{r\bar Vp\bar X} R_{U\bar s X}\,^p\Big] \\
&\quad -\frac{1}{2} \Phi \left[S^p_X R_{U\bar Vp\bar X}+S^{\bar q}_{\bar X}R_{U\bar V X\bar q} \right]\\
&\quad +  \left(S_{U\bar V}g_{X\bar X}+g_{U\bar V}S_{X\bar X}+S_{U\bar V}g_{X\bar X}+g_{U\bar V}S_{X\bar X}\right)\\
&\quad +  g^{r\bar s} B_{i\bar jX\bar X} T^i_{rU}T^{\bar j}_{\bar s\bar V}-\frac{1}{2} \left(B_{i\bar VX\bar X}S^i_U +B_{U\bar j X\bar X}S^{\bar j}_{\bar V} \right)\\
&\quad - \left(g^{r\bar s} \Phi_r \nabla_{\bar s}R_{i\bar jk\bar l}U^iV^{\bar j}X^k X^{\bar l}+g^{r\bar s} \Phi_{\bar s} \nabla_{r}R_{i\bar jk\bar l}U^iV^{\bar j}X^k X^{\bar l} \right)\\
&\quad -\left( g^{r\bar s}\Phi_r T^{\bar j}_{\bar s \bar V}R_{U\bar j X\bar X}+g^{r\bar s}\Phi_{\bar s} T^{i}_{r U}R_{i\bar V X\bar X}\right)+\Box \Phi \cdot R_{U\bar VX\bar X}.
\end{split}
\end{equation}
The equation of $\Box W_{V\bar UX\bar X}$ is similar. 

 Noted that we have $|\nabla Rm|\leq C(n,K)t^{-1/2}$ by Proposition \ref{Improved-est}. Therefore by combining \eqref{evo-Rm-1}, \eqref{doubletime}, \eqref{a2} and \eqref{a1} and using the property of $\phi$ and $r>>K$, we have
\begin{equation*}
\begin{split}
\Box |W_{U\bar VX\bar X}|^2
&\leq -|\nabla W_{U\bar VX\bar X}|^2-|\bar \nabla W_{U\bar VX\bar X}|^2+C_n\sqrt{\frac{K}{t}}W_{U\bar U}W_{V\bar V}\\
&\quad +\Phi \left |R^{\bar s}\,_{\bar V X}\,^p R_{U\bar s p\bar X} -R^{\bar s}\,_{\bar Vp\bar X} R_{U\bar s X}\,^p\right| |W_{V\bar UX\bar X}|\\
&\quad +\Phi \left |R^{\bar s}\,_{\bar U X}\,^p R_{V\bar s p\bar X} -R^{\bar s}\,_{\bar Up\bar X} R_{V\bar s X}\,^p\right| |W_{U\bar VX\bar X}|
\end{split}
\end{equation*}
The main trouble is the quadratic term appeared on the right hand side because of $\Phi$ there and $U,V$ take places at different curvature term.
\begin{equation}\label{cut-prob}
\begin{split}
&\quad \Phi \left |R^{\bar s}\,_{\bar V X}\,^p R_{U\bar s p\bar X} -R^{\bar s}\,_{\bar Vp\bar X} R_{U\bar s X}\,^p\right| \\
&\leq \frac{1}{\Phi}\left |W^{\bar s}\,_{\bar V X}\,^p W_{U\bar s p\bar X} -W^{\bar s}\,_{\bar Vp\bar X} W_{U\bar s X}\,^p\right|+C_nK\\
&\leq \left(\frac{C_n}{\Phi}+C_nK \right)\sqrt{W_{U\bar U}W_{V\bar V} }
\end{split}
\end{equation}
where we have used \eqref{a1} and \eqref{a2} in the last inequality. On the other hand, since at $(p,t_0)$, 
\begin{equation}
\begin{split}
10&\leq (20+c_2\sqrt{Kt}) W_{U\bar U }W_{V\bar V}\\
&= |\Phi R_{U\bar V X\bar X}-B_{U\bar V X\bar X}|^2\\
&\leq 8+2\Phi^2 K^2.
\end{split}
\end{equation}
Therefore, $\Phi \geq K^{-1}$ at $(p,t_0)$. Combines with \eqref{cut-prob}, 
\begin{equation}\label{cut-prob}
\begin{split}
 \Phi \left |R^{\bar s}\,_{\bar V X}\,^p R_{U\bar s p\bar X} -R^{\bar s}\,_{\bar Vp\bar X} R_{U\bar s X}\,^p\right| 
&\leq C_nK \sqrt{W_{U\bar U}W_{V\bar V}}.
\end{split}
\end{equation}

And hence at $(p,t_0)$, 
\begin{equation}
\begin{split}
\heat F&\leq 2(20+c_2\sqrt{Kt}){\bf Re}\left(g^{r\bar s}\nabla_r W_{U\bar U} \cdot \nabla_{\bar s}W_{V\bar V} \right)\\
&\quad  -|\nabla W_{U\bar VX\bar X}|^2-|\bar \nabla W_{U\bar VX\bar X}|^2\\
&\quad -\frac{c_2}{4}\sqrt{\frac{K}{t}} W_{U\bar U}W_{V\bar V}+2S_{X\bar X}|W_{U\bar VX\bar X}|^2
\end{split}
\end{equation}

By using the fact that $\nabla F=0$ and $F=0$ at $(p,t_0)$, one can conclude that
\begin{equation}
2(20+c_2\sqrt{Kt}){\bf Re}\left(g^{r\bar s}\nabla_r W_{U\bar U} \cdot \nabla_{\bar s}W_{V\bar V} \right) \leq |\nabla W_{U\bar VX\bar X}|^2+|\bar \nabla W_{U\bar VX\bar X}|^2.
\end{equation}

Using \eqref{doubletime} and \eqref{a2} again, we deduce that $$2S_{X\bar X}|W_{U\bar VX\bar X}|^2\leq C_nKW_{U\bar U}W_{V\bar V}$$ and hence at $(p,t_0)$, 
\begin{align}
\heat F<-\frac{c_2}{8}\sqrt{\frac{K}{t}} W_{U\bar U} W_{V\bar V}
\end{align}
which contradicts with \eqref{max-2} provided that $c_2(n)>>1$. This proves the claim.
\end{proof}
The assertion follows by letting $r_0\rightarrow\infty$.
\end{proof}

An immediate consequence is the following splitting theorem in \K case based on the strong maximum principle along the noncompact \KR flow and the De Rham decomposition theorem.
\begin{cor}
Let $g(t)$ be a complete solution to the \KR flow on a noncompact simply connected complex manifold $M^n$ with bounded curvature. If the initial metric $g_0$ has non-positive bisectional curvature. Then for sufficiently small $t>0$, either $g(t)$ has negative Ricci curvature on $M$ or $(M,g(t))$ splits holomorphically isometrically into a product $\mathbb{C}^k\times N^{n-k}$.
\end{cor}
\begin{proof}
Since we have established the preservation of curvature condition in Theorem \ref{preserve-Ric} in the noncompact case. By theorem \ref{preserve-Ric}, the tensor $e^{-At}Ric_{g(t)}$ satisfies the null vector condition for $A$ sufficiently large and $t$ sufficiently small. Then we can apply the standard Strong maximum principle argument in \cite[Theorem 12.50]{ChowRicciflow2} to show that the kernel of $Ric_{g(t)}$ is parallel in space and time, see also \cite[Page 1602-1603]{Liu2014} for the original argument in compact case. The flatness of the kernel follows by the second conclusion in Theorem \ref{preserve-Ric}.
\end{proof}

 As an application of Strong maximum principle, we have the following.
\begin{thm}\label{main-thm}
Suppose $(M,g(t)),\;t\in[0,\tau]$ is a complete solution to \eqref{HRF} satisfying the assumption in Theorem \ref{preserve-Ric}. If the initial metric $g_0$ has quasi-negative Chern-Ricci curvature, then $Ric(g(t))<0$ on $M\times (0,\tau']$ for some $\tau'>0$. In particular, $M$ supports a \K metric which is possibly incomplete.
\end{thm}
\begin{proof}
The argument is exactly the same as the argument in the  compact case \cite{Lee2018} except that we construct the barrier by solving Dirichlet problem on some compact set instead of the whole manifold $M$. We will closely follow the argument in \cite{ChowRicciflow2}. We here only point out the necessary modifications. Let $\tau'=\min\{c_1K^{-1},\tau\}$ be the number obtained from Theorem \ref{preserve-Ric}.

Let $y\in M$ be a point at which the Chern-Ricci curvature of $g_0$ is negative. For any $x\in M$, let $\Omega$ be a connected open set with smooth boundary and containing both $x$ and $y$. Let $\phi_0$ be a smooth nonnegative function such that $\phi_0(y)>0$, $\phi_0=0$ near $\partial \Omega$ and $$Ric(g_0)+\phi_0g_0\leq 0$$ on $\Omega$. Let $\phi(z,t)$ be the solution to the heat equation
\begin{equation}
\label{heat}
\begin{split}
\left( \frac{\partial}{\partial t}-\Delta_{g(t)}\right) \phi(x,t)&=0,\quad\text{on}\;\;\Omega\times [0,\tau'];\\
\phi(x,0)=\phi_0,\quad&\text{and}\quad \phi(x,t)|_{\partial\Omega}=0.
\end{split}
\end{equation}
It then follows by strong maximum principle that $\phi(x,t)>0$ on $\Omega\times(0,\tau']$. We may assume that $\phi(x,t)\leq 1$ by rescaling. As in \cite{Lee2018}, we consider the tensor 
$$A^\e=Ric_{g(t)}+e^{-kt}\phi^2 g(t)-\e e^{Bt}g(t)$$
where $B,k$ is some large constant. Then if $A^\e$ fail to be negative on $\overline{\Omega}\times [0,\tau']$, it can only happen at $(x_0,t_0)$ where $x_0\in int(\Omega)$ and $t_0>0$. Without loss of generality, we may assume $t_0$ to be the first time such that $A^\e$ fails to be negative. And we may apply second derivatives test at $t=t_0$.

Now the argument in \cite{Lee2018} can be carried over since the argument is purely local. Hence we can show that for any $\e>0$, $A^\e< 0$ on $\Omega\times [0,\tau']$. By letting $\e\rightarrow 0$, it shows that $Ric(x,t)<0$ for $t\in (0,\tau']$. Since $x$ is arbitrary, this completes the proof. By taking $h=-Ric(g(t_1))$, we see that $(M,h)$ is a \K manifold since the Chern-Ricci form is $d$-closed by definition.
\end{proof}

\section{Existence of \KE metric}
In this section, we will show that under the assumption in Theorem \ref{preserve-Ric}, if in addition $g_0$ has uniformly negative Chern-Ricci curvature outside a compact set, then $M$ supports a complete \KE metric with negative scalar curvature.

We first show that the uniform negativity at infinity will be preserved  along the flow with bounded Chern curvature and torsion.
\begin{prop}\label{inf-beh}
Suppose $(M,g(t)),\;t\in [0,\tau]$ is a complete solution to \eqref{HRF} with 
$$\sup_{M\times [0,\tau]}|Rm|+|T|^2\leq K_0$$
for some $K_0>0$. If there is $p\in M$, $r>0$ such that outside $B_{g_0}(p,r)$, the Chern-Ricci curvature $Ric(g_0)<-\delta$ for some $\delta>0$, then there is $\tilde \tau(n,K_0,\delta)>0$ so that on $[0,\tau]\cap [0,\tilde \tau]$, $Ric(g(t))<-\delta/2$ outside $B_{g_0}(p,r+1)$.
\end{prop}
\begin{proof}
The proof is standard. For the sake of completeness, we give the proof here. By rescaling, we may assume $\delta=1$. Let $z\in M$ such that $B_{g_0}(z,1)$ is disjoint from $B_{g_0}(p,r)$. Let $\phi$ be a cutoff function on $[0,+\infty)$ such that $\phi\equiv 1$ on $[0,1/2]$, vanishes outside $[0,1]$ and satisfies 
$$|\phi'|^2\leq 100\phi ,\;\; \phi''\geq -100\phi.$$
Let $\Phi=\phi(d_{g_0}(x,z))$ be a cutoff function on $M$. For $\e>0$, consider the $(1,1)$ type tensor $A^\e_{i\bar j}=\Phi(R_{i\bar j}+g_{i\bar j})-(\e+2L\sqrt{t}) g_{i\bar j}$ with $L>>1$. Our goal is to show that for any $\e>0$, $A^\e< 0$ for $t$ sufficiently small independent of $\e$. We will omit the index $\e$ for notational convenient. Clearly by continuity, it holds for $t< \tau(\e, g(t))$ sufficiently small. Let $t_0>0$ be the first such that $A$ fails to be negative. Then at $t=t_0$, there is $x\in B_{g_0}(z,1)$, $X_0\in T^{1,0}_xM$ such that 
$$A_{X_0\bar X_0}=0.$$ 
We may assume $|X_0|=1$ by rescaling. Extends $X_0$ to local vector field $X\in T^{1,0}M$ such that $\nabla X=0$ at $(p,t_0)$. Using the fact that $A_{X_0Y}=0$ for all $Y\in T^{1,0}M$ and the extension,
\begin{equation}
\begin{split}
0&\leq \heat A_{X\bar X}\\
&= \Box A_{i\bar j} \cdot X^i X^{\bar j}\\
&= \Box \left( \Phi R_{i\bar j}+\Phi g_{i\bar j} -\e g_{i\bar j} -2L\sqrt{t}g_{i\bar j} \right) \cdot X^i X^{\bar j}.
\end{split}
\end{equation}
Using the estimate on the cutoff function, Lemma \ref{dist-esti}, curvature assumptions, Proposition \ref{Improved-est} and Lemma \ref{Ricci-evo}, we have 
$$\heat A_{X\bar X} \leq c(n,K_0)t^{-1/2}-Lt^{-1/2}.$$

 Hence if $L>>1$ is large enough, we have got contradiction. By letting $\e\rightarrow 0$, we have shown that there is $L(n,K_0)>0$ such that for any $t\in [0,\tau]$, $x$ outside $B_{g_0}(p,r+1)$, 
$$R_{i\bar j} \leq (-1+L\sqrt{t})g_{i\bar j}.$$
The assertion follows by choosing $\tilde \tau$ small enough.
\end{proof}

Now we are ready to prove the existence of \KE metric.
\begin{thm}\label{KE-exist}
Suppose $(M,g_0)$ is a complete noncompact Hermitian manifold with
$$\sup_M \left\{ |Rm(g_0)|+|T_{g_0}|^2+|Rm^L(g_0)|\right\}<+\infty$$
 and non-positive bisectional curvature. If in addition there is a compact set $\Omega$, $\delta>0$ such that outside $\Omega$, $Ric(g_0)<-\delta g_0$. Then there is a  K\"ahler-Einstein metric $g_{KE}=-Ric(g_{KE})$ on $M$. Furthermore, the curvature tensor $Rm$ of $g_{KE}$ and all its covariant derivatives are bounded.
\end{thm}
\begin{proof}
By Theorem \ref{Improved-short-time}, there is a short-time solution $g(t)$ starting from $g(0)=g_0$ on $M\times [0,\tau]$ with bounded Chern curvature and torsion. By Theorem \ref{main-thm} and Proposition \ref{inf-beh}, there is another Hermitian metric $g(\tau)$ with $$-C g(\tau)\leq Ric(g(\tau))<-\sigma g(\tau)$$ for some $\sigma,C>0$. Let $h=-Ric(g(\tau))$. As the Chern-Ricci curvature is $d$-closed,  $h$ is a complete \K metric uniformly equivalent to $g(\tau)$. Moreover, the higher order estimate in Theorem \ref{Improved-short-time} implies that for all $m\in \mathbb{N}$, there is $C(n,m,K_0)>0$ such that 
\begin{align}\label{higher-order-est}|\nabla_{g(\tau)}^m h|\leq C(n,m,K_0).
\end{align}
Rewrite
$$-Ric(h)=-Ric(g(\tau)) +\ddb \log \frac{\det h}{\det g(\tau)}=h+\ddb F.$$

By \eqref{higher-order-est}, all the covarient derivatives of $F$ with respect to $h$ are bounded. By \cite[Theorem 5.1]{LottZhang2011}, there is a complete \KE metric $g_{KE}=-Ric(g_{KE})$ on $M$, see also \cite{ChauLee2019,HuangLeeTam2019}. Furthermore, by Shi's estimate \cite{Shi1989} or Proposition \ref{Improved-est}, the curvature tensor of $g_{KE}$ and all its covariant derivatives are bounded. 
\end{proof}

\end{document}